\documentclass{article}
\usepackage{amsmath,amsfonts,amssymb,amsthm,enumitem,algpseudocode,algorithm,fullpage,booktabs,hyperref,cleveref,physics,mathtools}

\usepackage{float}
\newfloat{algorithm}{t}{lop}

\newtheorem{lemma}{Lemma}
\newtheorem{theorem}{Theorem}
\newtheorem{corollary}{Corollary}

\theoremstyle{definition}

\theoremstyle{remark}

\author{Riley Badenbroek \and Joachim Dahl}
\title{An Algorithm for Nonsymmetric Conic Optimization Inspired by MOSEK}

\DeclareMathOperator{\diff}{d\!}
\DeclareMathOperator{\interior}{int}

\algnewcommand\Input{\item[\bf Input:] }
\algnewcommand\Output{\item[\bf Output:] }

\newcommand{\mue}{\mu^\text{e}}
\newcommand{\muetilde}{\tilde{\mu}^\text{e}}
\newcommand{\supaff}{\text{aff}}
\newcommand{\supcen}{\text{cen}}
\newcommand{\suppred}{\text{pred}}
\newcommand{\supcor}{\text{cor}}
\newcommand{\subone}{+}
\newcommand{\subtwo}{{++}}
\newcommand{\primbar}{F}
\newcommand{\primgrad}{F'}
\newcommand{\primhess}{F''}
\newcommand{\primtens}{F'''}
\newcommand{\dualbar}{F_*}
\newcommand{\dualgrad}{F_*'}
\newcommand{\dualhess}{F_*''}
\newcommand{\primbarext}{F^\text{e}}
\newcommand{\dualbarext}{F_*^\text{e}}
\newcommand{\dualgradext}{(F_*^\text{e})'}

\newcommand{\primdelta}{\delta^{\text{P}}}
\newcommand{\dualdelta}{\delta^{\text{D}}}
\newcommand{\primup}{u^{\text{P}}}
\newcommand{\dualup}{u^{\text{D}}}
\newcommand{\primlow}{l^{\text{P}}}
\newcommand{\duallow}{l^{\text{D}}}
\newcommand{\normub}{\theta}

\newcommand{\zinit}{z_0}
\newcommand{\xinit}{x_0}
\newcommand{\sinit}{s_0}
\newcommand{\yinit}{y_0}
\newcommand{\tauinit}{\tau_0}
\newcommand{\kappainit}{\kappa_0}

\begin{document}

\maketitle

\begin{abstract}
	We analyze the scaling matrix, search direction, and neighborhood used in MOSEK's algorithm for nonsymmetric conic optimization [Dahl and Andersen, 2019]. It is proven that these can be used to compute a near-optimal solution to the homogeneous self-dual model in polynomial time.
\end{abstract}


\section{Introduction}
In 1984, Karmarkar \cite{karmarkar1984new} introduced an interior point method to solve linear programming problems in polynomial time. 
Nesterov and Nemirovskii \cite{nesterov1994interior} generalized the ideas behind this method to convex optimization, recognizing that self-concordant barrier functions are fundamental to the analysis of interior point methods.

Meanwhile, researchers found that for linear optimization, primal-dual interior point methods generally outperformed their primal-only or dual-only counterparts. Naturally, they tried to extend these primal-dual methods to convex conic optimization.
A major breakthrough in this area is due to Nesterov and Todd \cite{nesterov1997self,nesterov1998primal}, who introduced search directions for self-scaled cones that perform well in practice.
G\"uler \cite{guler1996barrier} showed that these self-scaled cones are the same as symmetric cones, i.e. convex self-dual cones with a transitive automorphism group. There are five irreducible symmetric cones (up to isomorphism) \cite{jordan1934neumann}, only two of which are used for optimization in practice: the second-order cone, and the cone of symmetric positive semidefinite matrices. Since linear programming is a special case of semidefinite programming, it can also be written as an optimization problem over a symmetric cone.

However, the Nesterov-Todd directions are defined at a primal-dual feasible point. Hence, given an optimization problem, one still needs to find such a feasible point, if it even exists. To circumvent this issue, Ye, Todd, and Mizuno \cite{ye1994nl} introduced a homogeneous model (also known as a self-dual embedding) for linear programming. Two major advantages of this homogeneous model are that no strictly feasible starting point is required, and that the algorithm can generate certificates for primal or dual infeasibility. Generalizations of this homogeneous model were proposed by e.g. Luo, Sturm, and Zhang \cite{luo1996duality}, De Klerk, Roos, and Terlaky \cite{deklerk1997initialization}, Potra and Sheng \cite{potra1998homogeneous}, and Andersen and Ye \cite{andersen1999homogeneous}. Nesterov, Todd, and Ye \cite{nesterov1999infeasible} provide a complexity analysis of a class of algorithms to solve the homogeneous model for convex conic optimization, where the cones do not have to be symmetric.

A broader class of cones than the symmetric cones can be found by dropping the self-duality requirement: the resulting convex cones with a transitive automorphism group are called homogeneous. Although Chua \cite{chua2003relating} showed that each optimization problem over a homogeneous cone can be rewritten as a semidefinite optimization problem, Chua \cite{chua2009algebraic} also proposed algorithms to exploit this structure directly. An even broader class of cones is formed by the hyperbolicity cones, which G\"uler \cite{guler1997hyperbolic} showed contains the homogeneous cones. Moreover, G\"uler also proposed methods to solve optimization problems over these cones.

In this paper, we consider cones that are more general than the hyperbolicity cones. 
We only assume that the cones are proper and admit a logarithmically homogeneous self-concordant barrier.
These cones are referred to as nonsymmetric to emphasize we make few structural assumptions about them.
Interesting examples of nonsymmetric cones include the exponential cone 
\begin{equation*}
	\left\{ x \in \mathbb{R}^3 : x_1 \geq x_2 e^{x_3/x_2}, x_2 \geq 0 \right\},
\end{equation*}
and the three-dimensional power cone
\begin{equation*}
	\left\{ x \in \mathbb{R}^3 : x_1^a x_2^{1-a} \geq | x_3 | \right\},
\end{equation*}
for some $a \in (0,1)$.
Many convex optimization problems occurring in practice can be modeled using these two cones and the symmetric cones mentioned above. For example, Lubin et al. \cite{lubin2016extended} show that all convex problems from the MINLPLIB2 library can be expressed in this manner.
While the algorithms by Nesterov, Todd, and Ye \cite{nesterov1999infeasible} can be used to solve nonsymmetric problems, the size of the linear systems to be solved doubles compared to the symmetric case, increasing the computation time by a factor of eight.

A more recent approach by Nesterov \cite{nesterov2012towards} splits each iteration in two phases. In the correction phase, a strictly feasible primal-dual pair and an associated scaling point are computed. These are used in the prediction phase to find a primal-dual direction that is approximately tangential to the central path. This algorithm has the drawback that it assumes the existence of a strictly feasible primal-dual point, and requires a strictly feasible primal starting point.

Skajaa and Ye \cite{skajaa2015homogeneous} built on these two methods by proposing an algorithm that solves the homogeneous model and only uses the primal barrier, which means the size of the linear systems is the same as in the symmetric case. While theoretically attractive, this algorithm still requires centering steps. Serrano \cite{serrano2015algorithms} proposed a variant of Skajaa and Ye's algorithm that no longer uses centering steps, and implements this for the exponential cone in the ECOS solver. 

For two points $x$ and $s$ in a symmetric cone, we can always find a scaling point $w$ such that the Hessian of the cone's barrier at $w$ maps $x$ to $s$ and the gradient at $s$ to the gradient at $x$. Moreover, the Hessian at $w$ is bounded by the Hessians at $x$ and $s$ when $x$ and $s$ lie close to the central path.
For a general convex cone, we can define a similar scaling which maps $x$ to $s$, but generally not the gradient at $s$ to the gradient at $x$.
Thus, the main hurdle in generalizing the Nesterov-Todd directions to nonsymmetric cones is to find a mapping from $x$ to $s$ and the gradient at $s$ to the gradient at $x$, and to ensure that this mapping is in some sense close to the primal and dual Hessians to guarantee polynomial-time convergence. Tun\c{c}el \cite{tunccel2001generalization} proposed to form a such scaling matrix by low-rank updates to an arbitrary positive definite matrix.
Myklebust and Tun\c{c}el \cite{myklebust2016interior} provide explicit bounds on the scaling matrix if the current iterates $x$ and $s$ lie close to the central path.

Dahl and Andersen \cite{dahl2019primal} expand on the ideas from Skajaa and Ye \cite{skajaa2015homogeneous}, but define their search direction using the scaling matrices analyzed by Tun\c{c}el \cite{tunccel2001generalization} and Myklebust and Tun\c{c}el \cite{myklebust2016interior}. The resulting algorithm is implemented in MOSEK for the exponential and three-dimensional power cones. Dahl and Andersen provide no complexity analysis of their algorithm, but do show strong empirical performance.

The purpose of this work is to provide some theoretical foundation for the algorithm implemented in MOSEK for nonsymmetric conic optimization. Practical implementations will always differ from theoretical algorithms, so we focus on particular elements of Dahl and Andersen's algorithm: their scaling matrix, search direction, and neighborhood. With these ingredients, we will define an algorithm for nonsymmetric conic optimization that approximately solves the homogeneous model in polynomial time.

The remainder of this paper is structured as follows. After introducing the preliminaries in Section \ref{sec:Preliminaries}, we state the algorithm we will analyze in Section \ref{sec:Algorithm}. Moreover, in Section \ref{sec:Algorithm} we also introduce the assumptions that should hold at the start of each iteration. Section \ref{sec:ScalingMatrix} will cover the first step in our analysis: finding bounds on the scaling matrix as in Myklebust and Tun\c{c}el \cite{myklebust2016interior}. These bounds will be used in 
Sections \labelcref{sec:Predictor,sec:Corrector} 
to analyze what happens to the starting assumptions after our steps are taken. Section \ref{sec:Predictor} will focus on the predictor step, while the corrector step is analyzed in Section \ref{sec:Corrector}. Finally, we combine the results of the preceding sections in Section \ref{sec:Analysis} to prove our algorithm terminates in polynomial time.

\section{Preliminaries}
\label{sec:Preliminaries}
Throughout this paper, $\langle \cdot, \cdot \rangle$ will denote the dot inner product on $\mathbb{R}^n$, and $\| \cdot \|$ will be the Euclidean norm on $\mathbb{R}^n$. Let $\mathbb{S}^n \subset \mathbb{R}^{n \times n}$ denote the space of real symmetric matrices of size $n \times n$.
For a symmetric, positive definite matrix $S \in \mathbb{S}^{n}$, we define the induced norm
\begin{equation*}
	\| v \|_S \coloneqq \sqrt{\langle v, S v \rangle},
\end{equation*}
for any $v \in \mathbb{R}^n$. It is not hard to see that the dual norm to $\| \cdot \|_S$ satisfies
\begin{equation*}
	\| w \|_S^* \coloneqq \sup_{v : \| v \|_S \leq 1} |\langle v, w \rangle| = \| w \|_{S^{-1}},
\end{equation*}
for $w \in \mathbb{R}^n$. By definition of the dual norm, we have 
\begin{equation}
	\label{eq:CauchySchwarz}
	| \langle v, w \rangle | \leq \| v \|_S \| w \|_S^*,
\end{equation}
for any positive definite $S \in \mathbb{S}^{n}$.

We consider the primal conic optimization problem
\begin{equation}
	\label{eq:PrimalProblem}
	\inf_x \left\{ \langle c, x \rangle : A x = b, x \in K \right\}, \tag{P}
\end{equation}
and its dual
\begin{equation}
	\label{eq:DualProblem}
	\sup_{y,s} \left\{ \langle b, y \rangle : A^\top y + s = c, s \in K^* \right\}, \tag{D}
\end{equation}
where $c \in \mathbb{R}^n$, $A \in \mathbb{R}^{m \times n}$, $b \in \mathbb{R}^m$, $K \subset \mathbb{R}^n$ is a proper cone (i.e. convex, pointed, closed, and full-dimensional), and $K^* \coloneqq \{ s \in K^* : \langle x, s \rangle \geq 0 \,\forall x \in K \}$ is its dual cone.

\subsection{Self-Concordant Barriers}
Let $\primbar: \interior K \to \mathbb{R}$ be a thrice differentiable, strictly convex function. Through its Hessian $\primhess$, this function gives rise to a family of norms: for each $x \in \interior K$ and any $v \in \mathbb{R}^n$, we define $\| v \|_x \coloneqq \| v \|_{\primhess(x)} = \sqrt{\langle v, \primhess(x) v \rangle}$.
We assume that $\primbar$ is a logarithmically homogeneous self-concordant barrier (LHSCB) for $K$ with parameter $\nu$. This means $\primbar$ should satisfy
\begin{equation}
	\label{eq:DefHomogeneity}
	\primbar(t x) = \primbar(x) - \nu \log t,
\end{equation}
for all $t > 0$ and $x \in \interior K$, the Dikin ellipsoid $\{ x' \in \mathbb{R}^n : \| x' - x \|_x < 1 \} \subseteq \interior K$ for all $x \in \interior K$, and finally
\begin{equation}
	\label{eq:DefSelfConcordance}
	(1 - \| x' - x \|_x)^2 \primhess(x') \preceq \primhess(x) \preceq \frac{1}{(1-  \| x' - x \|_x)^2} \primhess(x'),
\end{equation}
for all $x, x' \in \interior K$ with $\| x' - x \|_x < 1$. We will call a LHSCB with parameter $\nu$ a $\nu$-LHSCB for short. It was shown by Nesterov and Nemirovskii \cite{nesterov1994interior} that any barrier parameter $\nu$ satisfies $\nu \geq 1$.

Standard consequences of the homogeneity \eqref{eq:DefHomogeneity} of $\primbar$ include that for all $x \in \interior K$,
\begin{equation}
	\label{eq:PropertiesLogHomogeneity}
	\primhess(x) x = - \primgrad(x) \qquad \text{and} \qquad \langle x, -\primgrad(x) \rangle = \nu.
\end{equation}

If $\primbar$ is a $\nu$-LHSCB for $K$, then the conjugate $\dualbar: \interior K^* \to\mathbb{R}$,
\begin{equation*}
	\dualbar(s) \coloneqq \sup_x \left\{ - \langle s, x\rangle - \primbar(x) : x \in \interior K \right\},
\end{equation*}
is a $\nu$-LHSCB for $K^*$. Similar to $\primbar$, $\dualbar$ also induces a family of norms through its Hessian $\dualhess$. For any $s \in \interior K^*$ and $v \in \mathbb{R}^n$, we define $\| v \|_s \coloneqq \| v \|_{\dualhess(s)} = \sqrt{\langle v, \dualhess(s) v \rangle}$. (It will be clear from context whether the local norm should depend on $\primhess$ or $\dualhess$.)
Moreover, it can be shown that $\{ -\primgrad(x) : x \in \interior K\} = \interior K^*$, and the gradient and Hessian of $\dualbar$ satisfy
\begin{equation}
	\label{eq:PropertiesConjugate}
	-\dualgrad(-\primgrad(x)) = x \qquad \text{and} \qquad \dualhess(-\primgrad(x)) = \primhess(x)^{-1},
\end{equation}
for any $x \in \interior K$. With this in mind, it makes sense to define the \emph{shadow iterates}
\begin{equation*}
	\tilde{x} \coloneqq - \dualbar(s) \in \interior K, \qquad \text{and} \qquad \tilde{s} \coloneqq - \primbar(x) \in \interior K^*.
\end{equation*}

It will prove interesting to measure the distance between $x$ and $\tilde{x}$, up to appropriate scaling. To this end, define
\begin{equation}
	\label{eq:DefMuMuTilde}
	\mu \coloneqq \frac{\langle x, s \rangle}{\nu} \qquad \text{and} \qquad \tilde{\mu} \coloneqq \frac{\langle \tilde{x}, \tilde{s} \rangle}{\nu}.
\end{equation}
As remarked by Tun{\c{c}}el \cite[Lemma 4.1]{tunccel2001generalization}, Nesterov and Todd \cite[Remark 1]{nesterov1998primal} showed that for any $x \in \interior K$ and $s \in \interior K^*$, we have
\begin{equation}
	\mu \tilde{\mu}\geq 1,
	\label{eq:MuTildeMuGreaterThanOne}
\end{equation}
with equality if and only if $x = \mu \tilde{x}$ (and hence $s = \mu \tilde{s}$).

\subsection{L\"owner Order and an Operator Norm}

For $P, Q \in \mathbb{S}^n$, we write $P \succeq Q$ to denote that $P-Q$ is positive semidefinite, and $P \succ Q$ to denote that $P-Q$ is positive definite. This partial ordering of $\mathbb{S}^n$ is referred to as the L\"owner order.
The following result by Horn and Johnson \cite{horn2012matrix} shows how the L\"owner order behaves when taking inverses.
\begin{lemma}[{\cite[Corollary 7.7.4(a)]{horn2012matrix}}]
	\label{lemma:OrderInverseMatrix}
	Let $P, Q \succ 0$ be positive definite matrices. Then, $P \preceq Q$ if and only if $P^{-1} \succeq Q^{-1}$.
\end{lemma}

We will sometimes try to ``sandwich'' a matrix between two other matrices, in the sense of the L\"owner order. To this end, let $P, Q \in \mathbb{S}^n$ where $P \succ 0$, and define the operator norm
\begin{equation}
	\label{eq:DefOperatorNorm}
	\| Q \|_P \coloneqq \sup_{\| u \|_P \leq 1} \| Q u \|_{P}^*.
\end{equation}
(This norm is also used by Myklebust and Tun\c{c}el \cite{myklebust2016interior}.) The reason for introducing this norm is given by the following lemma.
\begin{lemma}
	\label{lemma:NormToMatrixInequality}
	Let $P, Q \in \mathbb{S}^n$ where $P \succ 0$, and $\epsilon > 0$. Then, $\| P - Q \|_P \leq \epsilon$ if and only if
		 $(1-\epsilon) P \preceq Q \preceq (1+\epsilon) P$.
\end{lemma}
\begin{proof}
	Assume $\| P - Q \|_P \leq \epsilon$. To show $Q - (1-\epsilon)P \succeq 0$, it suffices to note that 
	by \eqref{eq:CauchySchwarz},
	\begin{align*}
		\inf_{\| u \|_P = 1} \langle u, \left(Q - (1-\epsilon)P \right) u \rangle
		&\geq \inf_{\| u \|_P = 1} \langle u, (Q-P) u \rangle + \inf_{\| u \|_P = 1} \langle u, \epsilon P u \rangle\\
		&\geq \inf_{\| u \|_P = 1} - \| u \|_P \| (Q-P) u \|_{P}^*  + \inf_{\| u \|_P = 1} \epsilon \| u \|_P^2 \\
		&= - \sup_{\| u \|_P = 1} \| (Q-P) u \|_{P}^* + \epsilon\\
		&\geq - \sup_{\| u \|_P \leq 1} \| (Q-P) u \|_{P}^* + \epsilon\\
		&\geq - \epsilon + \epsilon = 0.
	\end{align*}
	Therefore $\inf_{\| u \|_P \leq 1} \langle u, \left(Q - (1-\epsilon)P \right) u \rangle \geq 0$. Proving $(1+\epsilon)P - Q \succeq 0$ can be done similarly.
	
	Next, assume $(1-\epsilon) P \preceq Q \preceq (1+\epsilon) P$, i.e. $- \epsilon P \preceq P-Q \preceq \epsilon P$. Hence, for any $u \in \mathbb{R}^n$ with $\| u \|_P \leq 1$,
	we have $| \langle u, (P-Q) u \rangle | \leq \epsilon$. Using a transformation $v = P^{-1/2} u$,
	\begin{equation*}
		\epsilon \geq \sup_{u \neq 0} \frac{| \langle u, (P-Q) u \rangle |}{\| u \|_P^2} = \sup_{v \neq 0} \frac{| \langle v, P^{-1/2} (P-Q) P^{-1/2} v \rangle|}{\| v \|^2}.
	\end{equation*}
	It is well known (see e.g. Horn and Johnson \cite[Property 1.2.9]{horn1994topics}) that the supremum on the right hand side is attained at some eigenvector $v_*$ of $P^{-1/2} (P-Q) P^{-1/2}$ with norm one. Hence,
	\begin{equation*}
		\epsilon \geq \sup_{v \neq 0} \frac{| \langle v, P^{-1/2} (P-Q) P^{-1/2} v \rangle|}{\| v \|^2} = \sup_{v \neq 0} \frac{\| v \| \| P^{-1/2} (P-Q) P^{-1/2} v \|}{\| v \|^2} = \sup_{u \neq 0} \frac{\| (P-Q) u \|_{P}^*}{\| u \|_P}. \qedhere
	\end{equation*}
\end{proof}

We establish some straightforward properties of the norm \eqref{eq:DefOperatorNorm}. For any vectors $v,w$ and invertible linear operator $P$, we have
\begin{equation*}
	\| v w^\top \|_P = \sup_{\| u \|_P \leq 1} \| v w^\top u \|_{P}^* = \| v \|_{P}^* \sup_{\| u \|_P \leq 1} \langle w, u \rangle = \| v \|_{P}^* \| w \|_{P}^*,
\end{equation*}
and therefore, as Myklebust and Tun\c{c}el \cite{myklebust2016interior} showed,
\begin{equation}
	\label{eq:DifferenceOuterProductsNorm}
	\| vv^\top - ww^\top \|_P = \left\| \tfrac{1}{2} (v+w)(v-w)^\top + \tfrac{1}{2} (v-w)(v+w)^\top \right\|_P \leq \| (v+w)(v-w)^\top \|_P = \| v+w \|_{P}^* \| v-w \|_{P}^*.
\end{equation}
Moreover, for any $t > 0$,
\begin{equation}
	\label{eq:ScalingOperatorNorm}
	\| Q \|_{t P} = \sup_{\| u \|_{tP} \leq 1} \| Q u \|_{tP}^* = \sup_{\| \sqrt{t} u \|_{P} \leq 1} \| Q u / \sqrt{t}\|_{P}^* = \sup_{\| v \|_P \leq 1} \| Q v / t\|_{P}^* = \frac{1}{t} \| Q \|_P.
\end{equation}

\section{Algorithm Statement}
\label{sec:Algorithm}
In this section, we propose an algorithm to solve the pair \eqref{eq:PrimalProblem} and \eqref{eq:DualProblem}. In fact, we will solve a homogeneous model that will be introduced in Section \ref{subsec:HomogeneousModel}. The central path for this homogeneous model will be defined in Section \ref{subsec:CentralPath}, after which we outline the algorithm in Section \ref{subsec:SearchDirections}.

\subsection{Homogeneous Model}
\label{subsec:HomogeneousModel}
To solve \eqref{eq:PrimalProblem} and \eqref{eq:DualProblem}, we define the linear operator
\begin{equation*}
	G(y, x, \tau, s, \kappa) \coloneqq \begin{bmatrix}
	0 & A & -b\\
	-A^\top & 0 & c\\
	b^\top &-c^\top & 0
	\end{bmatrix} \begin{bmatrix} y \\ x \\ \tau \end{bmatrix} - \begin{bmatrix} 0 \\ s \\ \kappa \end{bmatrix}.
\end{equation*}
Then, the solutions to the homogeneous self-dual model, which are elements of the set
\begin{equation}
	\label{eq:HomogeneousModel}
	\left\{ (y, x, \tau, s, \kappa) \in \mathbb{R}^m \times K \times \mathbb{R}_+ \times K^* \times \mathbb{R}_+ : G(y, x, \tau, s, \kappa) = 0 \right\},
\end{equation}
have the following properties (see e.g. Skajaa and Ye \cite[Lemma 1]{skajaa2015homogeneous}):
\begin{enumerate}
	\item $\langle x, s \rangle + \tau \kappa = 0$;
	\item If $\tau > 0$, then $x/\tau$ is an optimal solution to \eqref{eq:PrimalProblem} and $(y,s)/\tau$ is an optimal solution to \eqref{eq:DualProblem};
	\item If $\kappa > 0$, then either $\langle b, y \rangle > 0$ or $\langle c, x \rangle < 0$, or both. If $\langle b, y \rangle > 0$, \eqref{eq:PrimalProblem} is infeasible. If $\langle c, x \rangle < 0$, \eqref{eq:DualProblem} is infeasible.
\end{enumerate}
In other words, finding an element of the set \eqref{eq:HomogeneousModel} where $\tau > 0$ or $\kappa > 0$ suffices to solve the primal-dual pair \eqref{eq:PrimalProblem} and \eqref{eq:DualProblem}. We will therefore interpret $G(y, x, \tau, s, \kappa)$ as the residual associated with the solution $(y, x, \tau, s, \kappa)$, which we would like to be the zero vector.

Assume we have initial points $\xinit \in \interior K$ and $\sinit \in \interior K^*$, such that $\xinit = -\dualgrad(\sinit)$, and hence $\sinit = -\primgrad(\xinit)$. 
MOSEK \cite{dahl2019primal} chooses $\xinit$ and $\sinit$, along with $\yinit$, $\tauinit$, and $\kappainit$ satisfying
\begin{equation}
	\label{eq:DefInitialPoint}
	\xinit = \sinit = -\primgrad(\xinit) = -\dualgrad(\sinit), \qquad \yinit = 0, \qquad \tauinit = \kappainit = 1,
\end{equation}
which admits a solution for the five cones that MOSEK supports. 
One of the perks of this choice is that $\langle \xinit, \sinit \rangle = \langle \xinit, - \primgrad(\xinit) \rangle = \nu$, meaning that the initial complementarity is known. 
We define the central path for the homogeneous model as
\begin{equation}
	\label{eq:CentralPath}
	\left\{ (y, x, \tau, s, \kappa) \in \mathbb{R}^m \times K \times \mathbb{R}_+ \times K^* \times \mathbb{R}_+ : \exists t \in (0, 1]: 
	\begin{array}{l}
		G(y, x, \tau, s, \kappa) = t G(\yinit, \xinit, \tauinit, \sinit, \kappainit) \\
		x = - t \dualgrad(s), s = - t \primgrad(x)\\
		\kappa \tau = t
	\end{array}
	\right\}.
\end{equation}
Informally, the condition $G(y, x, \tau, s, \kappa) = t G(\yinit, \xinit, \tauinit, \sinit, \kappainit)$ encodes that the residual norm should decrease as $t$ decreases, and the other conditions assure the ``centrality'' of the solution. We will discuss how to measure this centrality in the next section.

\subsection{The Central Path}
\label{subsec:CentralPath}
Suppose we have a $\nu$-LHSCB $\primbar$ for $K$ and a $\nu$-LHSCB $\dualbar$ for $K^*$. Since any element of \eqref{eq:HomogeneousModel} satisfies $(x, \tau) \in K \times \mathbb{R}_+$ and $(s, \kappa) \in K^* \times \mathbb{R}_+$, it will prove useful to extend our barriers $\primbar$ and $\dualbar$ to the domains $ K \times \mathbb{R}_+$ and $K^* \times \mathbb{R}_+$, respectively. A straightforward way to do this is to define
\begin{equation*}
	\primbarext(x, \tau) \coloneqq \primbar(x) - \log(\tau) \qquad \text{and} \qquad \dualbarext(s, \kappa) \coloneqq \dualbar(s) - \log(\kappa),
\end{equation*}
such that $\primbarext$ and $\dualbarext$ are $(\nu+1)$-LHSCBs for $K \times \mathbb{R}_+$ and $K^* \times \mathbb{R}_+$, respectively. 
(The superscript ``e'' refers to these ``extended'' cones.)
For these cones and barriers, the quantities analogous to \eqref{eq:DefMuMuTilde} are
\begin{equation*}
	\mue \coloneqq \frac{\langle x, s \rangle + \tau \kappa}{\nu + 1} \qquad \text{and} \qquad \muetilde \coloneqq \frac{\langle \tilde{x}, \tilde{s} \rangle + 1/(\tau \kappa)}{\nu + 1}.
\end{equation*}
Because $K \times \mathbb{R}_+$ and $K^* \times \mathbb{R}_+$ are each other's dual cones, the inequality \eqref{eq:MuTildeMuGreaterThanOne} carries over to this setting. To be explicit, we must have $\mue \muetilde \geq 1$, with equality if and only if $(x,\tau)$ equals
\begin{equation*}
	- \mue \dualgradext(s, \kappa) = \begin{bmatrix} - \mue \dualgrad(s) \\ \mue / \kappa \end{bmatrix} = \begin{bmatrix} \mue \tilde{x} \\ \mue / \kappa \end{bmatrix},
\end{equation*}
in which case we also have $s = \mue \tilde{s}$.
Thus, if $\mue \muetilde = 1$, the centrality conditions in \eqref{eq:CentralPath} are satisfied. Since in general $\mue \muetilde \geq 1$, we could define a neighborhood of the central path as the points satisfying $\beta \mue \muetilde \leq 1$ for some $\beta \in (0,1]$. For the sake of simplicity, MOSEK \cite[Section 3]{dahl2019primal} instead 
assumes
\begin{equation}
	\label{eq:MosekAssumptions}
	\frac{\beta}{\tau \kappa} \frac{\langle x, s \rangle + \tau \kappa}{\nu+1} \leq 1 \qquad \text{and} \qquad \beta \langle \tilde{x}, \tilde{s} \rangle \frac{\langle x, s \rangle + \tau \kappa}{\nu+1} \leq \nu,
\end{equation}
which implies
$\beta \mue \muetilde \leq 1$ for any $\beta \in (0,1]$. Some other useful properties of the assumptions \eqref{eq:MosekAssumptions} are given in the next lemma.

\begin{lemma}
	\label{lemma:MuMueBounds}
	Let $K \subset \mathbb{R}^n$ be a proper cone admitting a $\nu$-LHSCB $\primbar$. Let $x \in \interior K$, $s \in \interior K^*$, and $\tau, \kappa > 0$. Assume $\beta \mue \leq \tau \kappa$ and $\beta \mue \tilde{\mu} \leq 1$ for some $\beta \in (0,1]$. 
	Then,
	\begin{equation*}
		\frac{\mu}{2 - \beta} \leq 
		\mue \leq \frac{\mu}{\beta}.
	\end{equation*}
\end{lemma}
\begin{proof}
	Since $\tau \kappa \geq \beta \mue$,
	\begin{equation*}
		\mue = \frac{\mu \nu}{\nu + 1} + \frac{\tau \kappa}{\nu + 1} \geq \frac{\mu \nu}{\nu + 1} + \frac{\beta \mue}{\nu + 1},
	\end{equation*}
	which shows
	\begin{equation*}
		\mue \geq \left(1 - \frac{\beta}{\nu + 1} \right)^{-1} \frac{\mu \nu}{\nu + 1} = \frac{\mu \nu}{\nu + 1 - \beta} \geq \frac{\mu}{2 - \beta},
	\end{equation*}
	where the last inequality uses $\nu \geq 1$. 
	To prove the second part of the claim, we use the assumption that $\beta \tilde{\mu} \mue \leq 1$. By \eqref{eq:MuTildeMuGreaterThanOne}, we have $\tilde{\mu} \geq 1/\mu$, and thus
	\begin{equation*}
		1 \geq \beta \tilde{\mu} \mue \geq \frac{\beta \mue}{\mu},
	\end{equation*}
	which completes the proof.
\end{proof}
Thus, for high values of $\beta$, we have $\mu \approx \mue$ 
under MOSEK's assumptions \eqref{eq:MosekAssumptions},
and therefore $x \approx \mu \tilde{x}$ and $s \approx \mu \tilde{s}$. We will often refer to the distance between $x$ and $\mu \tilde{x}$, and between $s$ and $\mu \tilde{s}$, so we introduce the following shorthand notation:
\begin{equation*}
	\primdelta \coloneqq x - \mu \tilde{x} \qquad \text{and} \qquad \dualdelta \coloneqq s - \mu \tilde{s}.
\end{equation*}

\subsection{Search Directions}
\label{subsec:SearchDirections}
We now proceed to the statement of the algorithm that will be analyzed in the remainder of this paper. At the start of every iteration, we assume the following holds for some fixed $\beta \in (0,1]$ and $\eta \in [0, 1)$.
\begin{enumerate}[label=(A\arabic*)]
	\item \label{ass:XSinCone} $x \in \interior K$ and $s \in \interior K^*$
	\item \label{ass:TauKappaBeta} $\beta \mue \leq \tau \kappa$
	\item \label{ass:TauKappaPositive} $\tau, \kappa > 0$
	\item \label{ass:MuTildeBeta} $\beta \mue \tilde{\mu} \leq 1$
	\item \label{ass:PrimDelta} $\| \primdelta \|_x \leq \eta$.
\end{enumerate}
Note that \labelcref{ass:XSinCone,ass:TauKappaPositive,ass:TauKappaBeta,ass:MuTildeBeta} are also imposed by MOSEK \cite[Section 3]{dahl2019primal}. Assumption \ref{ass:PrimDelta} is important to ``sandwich'' the primal-dual scaling matrix used by MOSEK. This will be elaborated on in Section \ref{sec:ScalingMatrix}.

Every iteration consists of two phases. In the first phase, we apply a simplified version of the search direction used in MOSEK \cite{dahl2019primal}. The second phase consists of taking one corrector step to return to the assumptions \labelcref{ass:XSinCone,ass:TauKappaPositive,ass:TauKappaBeta,ass:MuTildeBeta,ass:PrimDelta}. For the sake of brevity, let us collect all variables in a vector $z \coloneqq (y, x, \tau, s, \kappa)$.

The first phase is started by computing a scaling matrix
\begin{equation}
	 \label{eq:DefScalingMatrix}
	 W \coloneqq \mu \primhess(x) + \frac{ss^\top}{\nu \mu} - \frac{\mu \tilde{s} \tilde{s}^\top}{\nu} + \frac{\dualdelta (\dualdelta)^\top}{\langle \primdelta, \dualdelta \rangle} - \frac{\mu[\primhess(x) \tilde{x} - \tilde{\mu} \tilde{s}] [\primhess(x)\tilde{x} - \tilde{\mu} \tilde{s}]^\top}{\| \tilde{x} \|_x^2 - \nu \tilde{\mu}^2}.
\end{equation}
We refer to $W$ as a scaling matrix because it is serves a similar purpose as the scaling point does for symmetric cones. Most importantly, we have
\begin{equation*}
	W x = s \qquad \text{and} \qquad W \tilde{x} = \tilde{s}.
\end{equation*}
Dahl and Andersen \cite[Section 5]{dahl2019primal} derive a Cholesky factorization of $W$, thereby showing $W$ is positive definite.

The scaling matrix $W$ will be used in the definition of the search direction in this first phase.
The search direction consists of two parts: an affine direction $\Delta z^\supaff$ and a centering direction $\Delta z^\supcen$. The affine direction is the solution $\Delta z^\supaff$ to
\begin{subequations}
	\begin{align}
		G(\Delta z^\supaff) &= -G(z) \label{eq:AffineDirectionG}\\
		\tau \Delta \kappa^\supaff + \kappa \Delta\tau^\supaff &= - \tau \kappa \label{eq:AffineDirectionTauKappa}\\
		W \Delta x^\supaff + \Delta s^\supaff &= -s. \label{eq:AffineDirectionXS}
	\end{align}
\end{subequations}
It follows from \eqref{eq:AffineDirectionG} that moving in the direction $\Delta z^\supaff$ decreases the norm of the residual, which amounts to progress in solving the self-dual homogeneous model. The centering direction $\Delta z^\supcen$ is the solution to
\begin{subequations}
	\begin{align}
		G(\Delta z^\supcen) &= G(z) \label{eq:CenteringDirectionG}\\
		\tau \Delta \kappa^\supcen + \kappa \Delta\tau^\supcen &= \mue \label{eq:CenteringDirectionTauKappa}\\
		W \Delta x^\supcen + \Delta s^\supcen &= \mue \tilde{s}. \label{eq:CenteringDirectionXS}
	\end{align}
\end{subequations}
We see from \eqref{eq:CenteringDirectionG} that moving in direction $\Delta z^\supcen$ increases the norm of the residual, but it serves to keep us 
close to the central path.
We will combine these two directions to get the search direction in the first phase: for some $\gamma \in [0,1]$ to be fixed later, let
\begin{equation}
	\label{eq:DefPredictor}
	\Delta z^\suppred \coloneqq \Delta z^\supaff + \gamma \Delta z^\supcen.
\end{equation}
Then, for some $\alpha \in (0,1]$ also to be fixed later, we update $z$ to the new iterate
\begin{equation*}
	z_\subone \coloneqq z + \alpha \Delta z^\suppred = z + \alpha (\Delta z^\supaff + \gamma \Delta z^\supcen).
\end{equation*}

After the predictor phase, we will need a corrector to be sure that \labelcref{ass:PrimDelta,ass:MuTildeBeta} hold.
To define it, we will compute a scaling matrix similar to \eqref{eq:DefScalingMatrix} defined as
\begin{equation}
	 \label{eq:DefScalingMatrixOne}
	 W_\subone \coloneqq \mu_\subone \primhess(x_\subone) + \frac{s_\subone s_\subone^\top}{\nu \mu_\subone} - \frac{\mu_\subone \tilde{s}_\subone \tilde{s}_\subone^\top}{\nu} + \frac{\dualdelta_\subone (\dualdelta_\subone)^\top}{\langle \primdelta_\subone, \dualdelta_\subone\rangle} - \frac{\mu_\subone[\primhess(x_\subone) \tilde{x}_\subone - \tilde{\mu}_\subone \tilde{s}_\subone] [\primhess(x_\subone)\tilde{x}_\subone - \tilde{\mu}_\subone \tilde{s}_\subone]^\top}{\| \tilde{x}_\subone \|_{x_\subone}^2 - \nu \tilde{\mu}_\subone^2},
\end{equation}
where all quantities with a subscript ``$\subone$'' follow their original definition, but computed for the iterate $z_\subone$ instead of $z$.
The corrector $\Delta z^\supcor_\subone$ will be the solution to
\begin{subequations}
	\begin{align}
		G(\Delta z^\supcor_\subone) &= 0 \label{eq:CorrectorDirectionG}\\
		\tau_\subone \Delta \kappa^\supcor_\subone + \kappa_\subone \Delta\tau_\subone^\supcor &= 0 \label{eq:CorrectorDirectionTauKappa}\\
		W_\subone \Delta x^\supcor_\subone + \Delta s^\supcor_\subone &= \mu_\subone \tilde{s}_\subone - s_\subone. \label{eq:CorrectorDirectionXS}
	\end{align}
\end{subequations}
It can be seen from \eqref{eq:CorrectorDirectionG} that $\Delta z^\supcor_\subone$ does not change the residuals, so the progress made by the predictor is maintained. 
(Skajaa and Ye \cite{skajaa2015homogeneous} also use a corrector that satisfies \eqref{eq:CorrectorDirectionG}, but they take -- in our notation -- $\tau_\subone^2 \Delta\kappa^\supcor_\subone + \mue_\subone \Delta\tau^\supcor_\subone = -\kappa_\subone \tau_\subone^2 + \mue_\subone \tau_\subone$ and $ \mu_\subone \primhess(x_\subone) \Delta x^\supcor_\subone + \Delta s^\supcor_\subone = \mue_\subone \tilde{s}_\subone - s_\subone $ instead of \labelcref{eq:CorrectorDirectionTauKappa,eq:CorrectorDirectionXS}.) 
We take one full corrector step to arrive at
\begin{equation*}
	z_\subtwo \coloneqq z_\subone + \Delta z^\supcor_\subone.
\end{equation*}
This $z_\subtwo$ will be the starting point for the next iteration.

The algorithm in this section is summarized in Algorithm \ref{alg:MOSEKNonSymmetric}.

\begin{algorithm}[ht!]
	\caption{An Algorithm for Nonsymmetric Conic Optimization (based on MOSEK \cite{dahl2019primal})}
	\label{alg:MOSEKNonSymmetric}
	\begin{algorithmic}[1]
		\Input Predictor step size $\alpha \in (0,1]$, corrector step size $\gamma \in [0,1]$.
		
		\State $z \gets (\yinit, \xinit, \tauinit, \sinit, \kappainit)$ as in \eqref{eq:DefInitialPoint}
		
		\While{not done}
			\State Compute scaling matrix $W$ as in \eqref{eq:DefScalingMatrix}
			\State Find the solution $\Delta z^\supaff$ to \labelcref{eq:AffineDirectionG,eq:AffineDirectionTauKappa,eq:AffineDirectionXS}, and the solution $\Delta z^\supcen$ to \labelcref{eq:CenteringDirectionG,eq:CenteringDirectionTauKappa,eq:CenteringDirectionXS}
			\State $z_\subone \gets z + \alpha \Delta z^\suppred = z + \alpha (\Delta z^\supaff + \gamma \Delta z^\supcen)$
			\State Compute scaling matrix $W_\subone$ as in \eqref{eq:DefScalingMatrixOne}
			\State Find the solution $\Delta z^\supcor_\subone$ to \labelcref{eq:CorrectorDirectionG,eq:CorrectorDirectionTauKappa,eq:CorrectorDirectionXS}
			\State $z \gets z_\subtwo = z_\subone + \Delta z^\supcor_\subone$
		\EndWhile
	\end{algorithmic}
\end{algorithm} 


\section{Scaling Matrix}
\label{sec:ScalingMatrix}
The scaling matrix $W$ is formed by low-rank updates to $\mu \primhess(x)$. We would like that $W \approx \mu \primhess(x)$ and $W \approx \frac{1}{\mu} \dualhess(s)^{-1}$ to derive further properties of Algorithm \ref{alg:MOSEKNonSymmetric}. This section is concerned with finding positive scalars $\primup$, $\primlow$, $\dualup$, and $\duallow$ such that
\begin{equation}
	\primlow \mu \primhess(x) \preceq W \preceq \primup \mu \primhess(x) \qquad \text{and} \qquad \frac{\duallow}{\mu} \dualhess(s)^{-1} \preceq W \preceq \frac{\dualup}{\mu} \dualhess(s)^{-1},
	\label{eq:ScalingHessianAssumption}
\end{equation}
and similarly, positive scalars $\primup_\subone$, $\primlow_\subone$, $\dualup_\subone$, and $\duallow_\subone$ such that
\begin{equation}
	\primlow_\subone \mu_\subone \primhess(x_\subone) \preceq W_\subone \preceq \primup_\subone \mu_\subone \primhess(x_\subone) \qquad \text{and} \qquad \frac{\duallow_\subone}{\mu_\subone} \dualhess(s_\subone)^{-1} \preceq W_\subone \preceq \frac{\dualup_\subone}{\mu_\subone} \dualhess(s_\subone)^{-1}.
	\label{eq:ScalingHessianAssumptionOne}
\end{equation}

For instance, Myklebust and Tun{\c{c}}el \cite[Theorem 6.8]{myklebust2016interior} derive such bounds for a different scaling matrix than the one by Dahl and Andersen \cite{dahl2019primal} under the condition $\| s - \mu \tilde{s} \|_s \leq 1/64$. 
Note that through Lemma \ref{lemma:OrderInverseMatrix}, \eqref{eq:ScalingHessianAssumption}  would also give us bounds on $W^{-1}$ in terms of $\frac{1}{\mu} \primhess(x)^{-1}$ and $\mu \dualhess(s)$, and \eqref{eq:ScalingHessianAssumptionOne} would give bounds on $W_\subone^{-1}$ in terms of $\frac{1}{\mu_\subone} \primhess(x_\subone)^{-1}$ and $\mu_\subone \dualhess(s_\subone)$.

We start with the following crucial lemma from Myklebust and Tun\c{c}el \cite{myklebust2016interior}, where some minor details are different in our setting.
\begin{lemma}[{\cite[Theorem B.1(4)]{myklebust2016interior}}]
	\label{lemma:DeltaDToDeltaP}
	Let $K \subset \mathbb{R}^n$ be a proper cone admitting a $\nu$-LHSCB $\primbar$. Let $x \in \interior K$ and $s \in \interior K^*$.
	If $\| \primdelta \|_x < 1$, then
	\begin{equation}
		\label{eq:DeltaDToDeltaP}
		\| \mu \primhess(x) \primdelta - \dualdelta \|_x^* \leq \frac{\mu \| \primdelta \|_x^2}{(1- \| \primdelta \|_x)^3}.
	\end{equation}
\end{lemma}
\begin{proof}
	Let $v \in \mathbb{R}^n$ be an arbitrary vector, and consider the function $\phi: [0,1] \to \mathbb{R}$ defined as
	\begin{equation*}
		\phi(t) := \langle v, \primgrad(x - t \primdelta) \rangle.
	\end{equation*}
	By Taylor's theorem, we have $\phi(1) = \phi(0) + \phi'(0) + \frac{1}{2} \phi''(r)$ for some $r \in [0,1]$.
	We first derive an upper bound on $\phi''(r)$ through self-concordance.
	Note that by definition
	\begin{equation*}
		\phi''(r) = \primtens(x - r \primdelta)[v,\primdelta, \primdelta] 
		= \lim_{q \to 0} \frac{|\langle \primdelta, [\primhess(x - r \primdelta + q v) - \primhess(x - r \primdelta)] \primdelta \rangle|}{q}.
	\end{equation*}
	By self-concordance, for sufficiently small $q$, we have
	\begin{equation*}
		\left( (1- q \| v \|_{x - r \primdelta})^2 - 1 \right) \primhess(x - r \primdelta) \preceq \primhess(x - r \primdelta + q v) - \primhess(x - r \primdelta)
		\preceq  \left( \frac{1}{(1- q \| v \|_{x - r \primdelta})^2} - 1 \right) \primhess(x - r \primdelta),
	\end{equation*}
	and therefore
	\begin{equation*}
		\phi''(r) \leq \lim_{q \to 0} \frac{1}{q} \left( \frac{1}{(1- q \| v \|_{x - r \primdelta})^2} - 1 \right) \| \primdelta \|_{x - r \primdelta}^2
		= 2 \| v \|_{x - r \primdelta} \| \primdelta \|_{x - r \primdelta}^2 \leq \frac{2 \| v \|_{x} \| \primdelta \|_{x}^2}{ (1 - \| \primdelta \|_x)^3 },
	\end{equation*}
	where the final inequality uses self-concordance and $r \leq 1$. Hence,
	\begin{equation*}
		\langle v, \primhess(x) \primdelta - s/\mu + \tilde{s} \rangle = -\phi'(0) + \phi(1) - \phi(0) = \tfrac{1}{2} \phi''(r) \leq \frac{ \| v \|_{x} \| \primdelta \|_{x}^2}{ (1 - \| \primdelta \|_x)^3 }.
	\end{equation*}
	With this upper bound, we can show
	\begin{equation*}
		\| \mu \primhess(x) \primdelta - \dualdelta \|_x^* = \sup_{v : \| v \|_x \leq 1} \langle v, \mu \primhess(x) \primdelta - \dualdelta \rangle \leq \sup_{v : \| v \|_x \leq 1} \mu \| v \|_{x} \frac{ \| \primdelta \|_{x}^2}{ (1 - \| \primdelta \|_x)^3 },
	\end{equation*}
	which proves the claim.
\end{proof}

We are now ready to derive bounds on $W$ as in \eqref{eq:ScalingHessianAssumption},
roughly following the approach by Myklebust and Tun\c{c}el \cite{myklebust2016interior}. We make the assumption $\| \primdelta \|_x \leq 0.18226$ to ensure that all denominators in these bounds are positive. The reader is suggested to skip the proof on a first pass.

\begin{theorem}[\cite{myklebust2016interior}]
	\label{thm:ScalingBounds}
	Let $K \subset \mathbb{R}^n$ be a proper cone admitting a $\nu$-LHSCB $\primbar$. Let $x \in \interior K$ and $s \in \interior K^*$, and assume $\| \primdelta \|_x \leq 0.18226$.
	Let $W$ be defined as in \eqref{eq:DefScalingMatrix}. 
	Then, the assumptions \eqref{eq:ScalingHessianAssumption} are satisfied with values
	\begin{equation*}
		\primlow = 1 - \epsilon_1 - \epsilon_2, \quad \primup = 1+\epsilon_1 + \epsilon_2, \quad \duallow = (1-\epsilon_1-\epsilon_2)(1 - \|\primdelta \|_x)^2, \quad \dualup = \frac{1 + \epsilon_1 + \epsilon_2}{(1 - \|\primdelta \|_x)^2},
	\end{equation*}
	where
	\begin{align*}
		\epsilon_1 &\coloneqq \frac{1}{\nu} \left( \| \primdelta \|_x  + \frac{\| \primdelta \|_x^2}{(1- \|\primdelta \|_x)^3} \right) \left( \| \primdelta \|_x  + \frac{ \| \primdelta \|_x^2}{(1- \|\primdelta \|_x)^3} + 2 \sqrt{\nu} \right)\\
		\epsilon_2 &\coloneqq \frac{2}{(1 - \| \primdelta \|_x)^3 - \| \primdelta \|_x}
		\left( \frac{4 \| \primdelta \|_x^2}{(1- \|\primdelta \|_x)^3} + 2 \| \primdelta \|_x
		+ \frac{\left( \frac{3 \| \primdelta \|_x^2}{(1- \|\primdelta \|_x)^3} +  \| \primdelta \|_x \right)^2}{ \| \primdelta \|_x \left( 1 - \frac{3 \| \primdelta \|_x}{ (1 - \| \primdelta \|_x)^3 } \right)} \right).
	\end{align*}
\end{theorem}
\begin{proof}
	We will try to bound $W - \mu \primhess(x)$ in two steps, each considering two terms from \eqref{eq:DefScalingMatrix}.
	First, observe that by \eqref{eq:DifferenceOuterProductsNorm},
	\begin{equation}
		\label{eq:FirstBoundFirstScalingTerm}
		\left\| \frac{ss^\top}{\nu \mu} - \frac{\mu \tilde{s} \tilde{s}^\top}{\nu} \right\|_{\primhess(x)} \leq \frac{\| s - \mu \tilde{s} \|_x^* \| s + \mu \tilde{s} \|_x^*}{\mu \nu} \leq \frac{\| \dualdelta \|_x^* (\|\dualdelta \|_x^* + 2 \mu \| \tilde{s} \|_x^*)}{\mu \nu},
	\end{equation}
	where $\| \tilde{s} \|_x^* = \| \tilde{s} \|_{\dualhess(\tilde{s})} = \sqrt{\nu}$ by \eqref{eq:PropertiesConjugate} and \eqref{eq:PropertiesLogHomogeneity}. Moreover, by Lemma \ref{lemma:DeltaDToDeltaP},
	\begin{equation}
		\label{eq:BoundDeltaD}
		\| \dualdelta \|_x^* \leq \| \mu \primhess(x) \primdelta \|_x^* + \| \mu \primhess(x) \primdelta - \dualdelta \|_x^*
		\leq \mu \| \primdelta \|_x  + \frac{\mu \| \primdelta \|_x^2}{(1- \|\primdelta \|_x)^3},
	\end{equation}
	and therefore, by using \eqref{eq:ScalingOperatorNorm}, \eqref{eq:FirstBoundFirstScalingTerm} can be further developed to
	\begin{equation}
		\label{eq:BoundFirstScalingTerm}
		\left\| \frac{ss^\top}{\nu \mu} - \frac{\mu \tilde{s} \tilde{s}^\top}{\nu} \right\|_{\mu \primhess(x)} \leq \frac{1}{\nu} \left( \| \primdelta \|_x  + \frac{\| \primdelta \|_x^2}{(1- \|\primdelta \|_x)^3} \right) \left( \| \primdelta \|_x  + \frac{ \| \primdelta \|_x^2}{(1- \|\primdelta \|_x)^3} + 2 \sqrt{\nu} \right) = \epsilon_1.
	\end{equation}
	
	For the second step, let
	\begin{equation*}
		H := \mu \primhess(x) + \frac{ss^\top}{\nu \mu} - \frac{\mu \tilde{s} \tilde{s}^\top}{\nu},
	\end{equation*}
	be a rank-two update to $\mu \primhess(x)$ such that $H x = s$ but $H \tilde{x} \neq \tilde{s}$ in general. Then, as was noted in \cite[Theorem 6.6]{myklebust2016interior}, we can write
	\begin{equation}
		\frac{\dualdelta (\dualdelta)^\top}{\langle \primdelta, \dualdelta \rangle} - \frac{\mu[\primhess(x)\tilde{x} - \tilde{\mu} \tilde{s}] [\primhess(x)\tilde{x} - \tilde{\mu} \tilde{s}]^\top}{\| \tilde{x} \|_x^2 - \nu \tilde{\mu}^2}
		= \frac{\dualdelta (\dualdelta)^\top - H \primdelta (\primdelta)^\top H}{\langle \primdelta, \dualdelta \rangle} + \left( \frac{1}{\langle \primdelta, \dualdelta \rangle} - \frac{1}{\| \primdelta \|_{H}^2} \right) H \primdelta (\primdelta)^\top H.
		\label{eq:RewriteTwoTermsH}
	\end{equation}
	Since
	\begin{equation*}
		\frac{1}{\langle \primdelta, \dualdelta \rangle} - \frac{1}{\| \primdelta \|_{H}^2} = \frac{\langle \primdelta, H \primdelta - \dualdelta \rangle}{\langle \primdelta, \dualdelta \rangle \| \primdelta \|_{H}^2 } \leq \frac{\| \primdelta\|_x \| H \primdelta - \dualdelta \|_x^*}{\langle \primdelta, \dualdelta \rangle \| \primdelta \|_{H}^2 },
	\end{equation*}
	it will suffice to know upper bounds on $\| \dualdelta - H \primdelta \|_x^*$, $\| \dualdelta + H \primdelta \|_x^*$, and $\| H \primdelta \|_x^*$, as well as lower bounds on $\langle \primdelta, \dualdelta \rangle$ and $\| \primdelta \|_{H}^2$. We derive these bounds in a manner similar to \cite[Lemma B.4]{myklebust2016interior}, although minor details again differ.
	
	Using Lemma \ref{lemma:DeltaDToDeltaP}, is it easy to show, similar to \cite[Lemma B.4(3)]{myklebust2016interior},
	\begin{align}
		\langle \primdelta, \dualdelta \rangle
		&= \langle \primdelta, \mu \primhess(x)\primdelta - \mu \primhess(x) \primdelta + \dualdelta \rangle \nonumber \\
		&\geq \mu \| \primdelta \|_x^2 - \| \primdelta \|_x \| \mu \primhess(x) \primdelta - \dualdelta \|_x^* \nonumber\\
		&\geq \mu \| \primdelta \|_x^2 \left( 1 - \frac{\| \primdelta \|_x}{ (1 - \| \primdelta \|_x)^3 } \right). \label{eq:BoundDeltaXDeltaS}
	\end{align}
	
	Moreover, as in \cite[Lemma B.4(7)]{myklebust2016interior},
	\begin{equation*}
		|\langle x, \mu \primhess(x) \primdelta \rangle| \leq |\langle x, \dualdelta \rangle| + \| x\|_x \| \mu \primhess(x) \primdelta - \dualdelta \|_x^* = 0 + \sqrt{\nu} \| \mu \primhess(x) \primdelta - \dualdelta \|_x^*,
	\end{equation*}
	where we used $\langle x, \dualdelta \rangle = \langle x, s - \mu \tilde{s} \rangle = \mu \nu - \mu \nu$ and \eqref{eq:PropertiesLogHomogeneity}.
	Thus, Lemma \ref{lemma:DeltaDToDeltaP} shows (cf. \cite[Lemma B.4(8)]{myklebust2016interior})
	\begin{align}
		\| H \primdelta - \dualdelta \|_x^*
		&= \left\| \mu \primhess(x) \primdelta - \dualdelta  + \frac{\mu \langle \tilde{s}, \primdelta \rangle}{\nu} \tilde{s} \right\|_x^* \nonumber\\
		&\leq \| \mu \primhess(x) \primdelta - \dualdelta \|_x^* + \frac{|\langle x, \mu \primhess(x) \primdelta \rangle|}{\nu} \| \tilde{s} \|_{\dualhess(\tilde{s})} \nonumber\\
		&\leq \frac{2 \mu \| \primdelta \|_x^2}{(1- \|\primdelta \|_x)^3}.
		\label{eq:BoundHDeltaXDeltaS}
	\end{align}
	Combining \eqref{eq:BoundDeltaXDeltaS} and \eqref{eq:BoundHDeltaXDeltaS} as in \cite[Lemma B.4(11)]{myklebust2016interior},
	\begin{align*}
		\| \primdelta \|_H^2
		&= \langle \primdelta, \dualdelta \rangle + \langle \primdelta, H \primdelta - \dualdelta \rangle\\
		&\geq \mu \| \primdelta \|_x^2 \left( 1 - \frac{\| \primdelta \|_x}{ (1 - \| \primdelta \|_x)^3 } \right) - \frac{2 \mu \| \primdelta \|_x^3}{(1- \|\primdelta \|_x)^3}\\
		&= \mu \| \primdelta \|_x^2 \left( 1 - \frac{3 \| \primdelta \|_x}{ (1 - \| \primdelta \|_x)^3 } \right).
	\end{align*}
	
	Finally, as in \cite[Lemma B.4(9-10)]{myklebust2016interior}, it follows from \eqref{eq:BoundDeltaD} and \eqref{eq:BoundHDeltaXDeltaS} that
	\begin{equation*}
		\| H \primdelta \|_x^* \leq \| H \primdelta - \dualdelta \|_x^* + \| \dualdelta \|_x^* \leq \frac{3\mu \| \primdelta \|_x^2}{(1- \|\primdelta \|_x)^3} + \mu \| \primdelta \|_x,
	\end{equation*}
	and similarly,
	\begin{equation*}
		\| H \primdelta + \dualdelta \|_x^* \leq \frac{4 \mu \| \primdelta \|_x^2}{(1- \|\primdelta \|_x)^3} + 2 \mu \| \primdelta \|_x,
	\end{equation*}
	
	We now have the tools to bound the operator norm of \eqref{eq:RewriteTwoTermsH}. By \eqref{eq:DifferenceOuterProductsNorm},
	\begin{align*}
		&\left\| \frac{\dualdelta (\dualdelta)^\top - H \primdelta (\primdelta)^\top H}{\langle \primdelta, \dualdelta \rangle} + \left( \frac{1}{\langle \primdelta, \dualdelta \rangle} - \frac{1}{\| \primdelta \|_{H}^2} \right) H \primdelta (\primdelta)^\top H \right\|_{\primhess(x)}\\
		&\leq \frac{ \| H \primdelta - \dualdelta \|_x^* \| H \primdelta + \dualdelta \|_x^* }{\langle \primdelta, \dualdelta \rangle}
		+ \frac{\| \primdelta\|_x \| H \primdelta - \dualdelta \|_x^*}{\langle \primdelta, \dualdelta \rangle \| \primdelta \|_{H}^2 } (\| H \primdelta \|_x^*)^2\\
		&= \frac{ \| H \primdelta - \dualdelta \|_x^*}{\langle \primdelta, \dualdelta \rangle} \left( \| H \primdelta + \dualdelta \|_x^* + \| \primdelta \|_x \frac{(\| H \primdelta \|_x^*)^2}{\| \primdelta \|_H^2} \right)\\
		&\leq \frac{\frac{2 \mu \| \primdelta \|_x^2}{(1- \|\primdelta \|_x)^3}}{\mu \| \primdelta \|_x^2 \left( 1 - \frac{\| \primdelta \|_x}{ (1 - \| \primdelta \|_x)^3 } \right)}
		\left( \frac{4 \mu \| \primdelta \|_x^2}{(1- \|\primdelta \|_x)^3} + 2 \mu \| \primdelta \|_x
		+ \frac{\left( \frac{3\mu \| \primdelta \|_x^2}{(1- \|\primdelta \|_x)^3} + \mu \| \primdelta \|_x \right)^2}{\mu \| \primdelta \|_x \left( 1 - \frac{3 \| \primdelta \|_x}{ (1 - \| \primdelta \|_x)^3 } \right)} \right) = \epsilon_2 \mu.
	\end{align*}
	Then, by \eqref{eq:ScalingOperatorNorm},
	\begin{equation}
		\label{eq:BoundSecondScalingTerm}
		\left\| \frac{\dualdelta (\dualdelta)^\top - H \primdelta (\primdelta)^\top H}{\langle \primdelta, \dualdelta \rangle} + \left( \frac{1}{\langle \primdelta, \dualdelta \rangle} - \frac{1}{\| \primdelta \|_{H}^2} \right) H \primdelta (\primdelta)^\top H \right\|_{\mu \primhess(x)} \leq \epsilon_2.
	\end{equation}
	Hence, we see from the triangle inequality, \eqref{eq:BoundFirstScalingTerm}, and \eqref{eq:BoundSecondScalingTerm} that $\| W - \mu \primhess(x) \|_{\mu \primhess(x)} \leq \epsilon_1 + \epsilon_2$. Lemma \ref{lemma:NormToMatrixInequality} therefore shows
	\begin{equation*}
		(1-\epsilon_1-\epsilon_2) \mu \primhess(x) \preceq W \preceq (1+\epsilon_1 + \epsilon_2) \mu \primhess(x).
	\end{equation*}
	Finally, we can use \eqref{eq:DefSelfConcordance} to bound $\frac{1}{\mu} \dualhess(s)^{-1} = \mu \primhess(\mu \tilde{x})$ as
	\begin{equation*}
		(1-\epsilon_1-\epsilon_2)(1 - \|\primdelta \|_x)^2 \tfrac{1}{\mu} \dualhess(s)^{-1} \preceq W \preceq \frac{1+\epsilon_1 + \epsilon_2}{(1 - \|\primdelta \|_x)^2} \tfrac{1}{\mu} \dualhess(s)^{-1},
	\end{equation*}
	similar to \cite[Theorem 6.8]{myklebust2016interior}.
\end{proof}
Of course, bounds on $W_\subone$ in terms of $\mu_\subone \primhess(x_\subone)$ and $\frac{1}{\mu_\subone} \dualhess(s_\subone)$ can also be found using Theorem \ref{thm:ScalingBounds} by replacing $\| \primdelta \|_x$ by $\| \primdelta_\subone \|_{x_\subone}$ in the definition of $\epsilon_1$ and $\epsilon_2$.
For ease of reference, we state these bounds on $W_\subone$ separately.
\begin{corollary}
	\label{cor:ScalingBoundsOne}
	Let $K \subset \mathbb{R}^n$ be a proper cone admitting a $\nu$-LHSCB $\primbar$. Let $x_\subone \in \interior K$ and $s_\subone \in \interior K^*$, and assume $\| \primdelta_\subone \|_{x_\subone} \leq 0.18226$.
	Let $W_\subone$ be defined as in \eqref{eq:DefScalingMatrixOne}. 
	Then, the assumptions \eqref{eq:ScalingHessianAssumptionOne} are satisfied with values
	\begin{equation*}
		\primlow_\subone = 1 - \epsilon_3 - \epsilon_2, \quad \primup_\subone = 1+\epsilon_3 + \epsilon_4, \quad \duallow_\subone = (1-\epsilon_3-\epsilon_4)(1 - \|\primdelta_\subone \|_{x_\subone})^2, \quad \dualup_\subone = \frac{1 + \epsilon_3 + \epsilon_4}{(1 - \| \primdelta_\subone \|_{x_\subone})^2},
	\end{equation*}
	where
	\begin{align*}
		\epsilon_3 &\coloneqq \frac{1}{\nu} \left( \| \primdelta_\subone \|_{x_\subone}  + \frac{\| \primdelta_\subone \|_{x_\subone}^2}{(1- \| \primdelta_\subone \|_{x_\subone})^3} \right) \left( \| \primdelta_\subone \|_{x_\subone}  + \frac{ \| \primdelta_\subone \|_{x_\subone}^2}{(1- \| \primdelta_\subone \|_{x_\subone})^3} + 2 \sqrt{\nu} \right)\\
		\epsilon_4 &\coloneqq \frac{2}{(1 - \| \primdelta_\subone \|_{x_\subone})^3 - \| \primdelta_\subone \|_{x_\subone}}
		\left( \frac{4 \| \primdelta_\subone \|_{x_\subone}^2}{(1- \| \primdelta_\subone \|_{x_\subone})^3} + 2 \| \primdelta_\subone \|_{x_\subone}
		+ \frac{\left( \frac{3 \| \primdelta_\subone \|_{x_\subone}^2}{(1- \| \primdelta_\subone \|_{x_\subone})^3} +  \| \primdelta_\subone \|_{x_\subone} \right)^2}{ \| \primdelta_\subone \|_{x_\subone} \left( 1 - \frac{3 \| \primdelta_\subone \|_{x_\subone}}{ (1 - \| \primdelta_\subone \|_{x_\subone})^3 } \right)} \right).
	\end{align*}
\end{corollary}

\section{Properties of the Predictor}
\label{sec:Predictor}

Now that we know that the scaling matrix $W$ is approximately equal to $\mu \primhess(x)$ and $\frac{1}{\mu} \dualhess(s)^{-1}$, we shift our focus to the predictor direction that it defines. 
Some simple properties of the predictor are derived in Section \ref{subsec:PredictorProperties}, followed by an upper bound on the local norm of the predictor in Section \ref{subsec:PredictorNorm}. Thus, we will be able to derive sufficient conditions for \ref{ass:XSinCone} to hold after the predictor step. We then consider what happens to assumptions \labelcref{ass:TauKappaBeta,ass:TauKappaPositive} in Section \ref{subsec:PredictorTauKappa}. It turns out we do not need to consider \ref{ass:MuTildeBeta}, so we conclude with an analysis of \ref{ass:PrimDelta} in Section \ref{subsec:PredictorShadowDistance}.

In Section \ref{sec:Corrector}, the effect of the corrector on all assumptions will be analyzed. To clarify the structure of our argument, an overview of what assumption will be analyzed where is given in Table \ref{tab:ArgumentStructure}.
\begin{table}
	\centering
	\begin{tabular}{lll}
		\toprule
		& \multicolumn{2}{c}{Status after} \\
		\cmidrule{2-3}
		Assumption & Predictor & Corrector\\
		\midrule
		\ref{ass:XSinCone}: $x \in \interior K$ and $s \in \interior K^*$ & Theorem \ref{thm:PredictorNorms} & Lemma \ref{lemma:CorrectorProperties}\ref{item:CorrectorNorm} \\
		\ref{ass:TauKappaBeta}: $\beta \mue \leq \tau \kappa$ & Lemma \ref{lemma:PredictorTauKappa} & Lemma \ref{lemma:CorrectorProperties}\ref{item:CorrectorTauKappa}\\
		\ref{ass:TauKappaPositive}: $\tau, \kappa > 0$ & Lemma \ref{lemma:PredictorTauKappa} & Lemma \ref{lemma:CorrectorProperties}\ref{item:CorrectorTauKappa}\\
		\ref{ass:MuTildeBeta}: $\beta \mue \tilde{\mu} \leq 1$ & -- & Lemma \ref{lemma:ShadowInnerProduct} \\
		\ref{ass:PrimDelta}: $\| \primdelta \|_x \leq \eta$ & Lemma \ref{lemma:PrimDeltaOne} & Lemma \ref{lemma:PrimDeltaTwo}\\
		\bottomrule
	\end{tabular}
	\caption{The results where the effect of the predictor and corrector steps on the assumptions are (mainly) analyzed}
	\label{tab:ArgumentStructure}
\end{table}

\subsection{Residual Reduction and Complementarity}
\label{subsec:PredictorProperties}
As we noted in Section \ref{subsec:SearchDirections}, the affine direction $\Delta z^\supaff$ decreases the norm of the residuals, and the centering direction $\Delta z^\supcen$ increases it. Dahl and Andersen \cite[Lemma 3]{dahl2019primal} show by how much the residuals decrease, and what the value of $\mue_\subone$ is. For the sake of completeness, we prove this result using our notation, and we moreover show that $\langle \Delta x^\suppred, \Delta s^\suppred \rangle + \Delta \tau^\suppred \Delta \kappa^\suppred$ is zero.

\begin{lemma}
	\label{lemma:PredictorProperties}
	Let $K \subset \mathbb{R}^n$ be a proper cone admitting a $\nu$-LHSCB $\primbar$.
	Pick $z = (y, x, \tau, s, \kappa)$ such that \labelcref{ass:XSinCone,ass:TauKappaPositive} hold, and let $\alpha, \gamma \in \mathbb{R}$. 
	Then, the following properties hold:
	\begin{enumerate}[label=(\roman*)]
		\item \label{item:PredictorG} $G(z + \alpha \Delta z^\suppred) = (1 - \alpha(1-\gamma)) G(z)$
		\item \label{item:PredictorOrthogonal} $\langle \Delta x^\suppred, \Delta s^\suppred \rangle + \Delta \tau^\suppred \Delta \kappa^\suppred = 0$
		\item \label{item:PredictorComplementarity} $\langle x + \alpha \Delta x^\suppred, s + \alpha \Delta s^\suppred \rangle + (\tau + \alpha \Delta \tau^\suppred ) (\kappa + \alpha \Delta \kappa^\suppred) = (1 - \alpha(1-\gamma))[\langle x,s \rangle + \tau \kappa]$. In other words, $\mue_\subone = (1 - \alpha(1-\gamma)) \mue$.
	\end{enumerate}
\end{lemma}
\begin{proof}
	\ref{item:PredictorG}: Follows directly from \eqref{eq:AffineDirectionG}, \eqref{eq:CenteringDirectionG}, and the fact that $G$ is a linear operator.
	
	\ref{item:PredictorOrthogonal}: Using \ref{item:PredictorG} and the linearity of $G$, we have $0 = G(z+\alpha \Delta z^\suppred - (1-\alpha(1-\gamma))z) = \alpha G(\Delta z^\suppred + (1-\gamma)z)$. By skew-symmetry,
	\begin{align*}
		0 &= \left\langle \begin{bmatrix} \Delta y^\suppred + (1-\gamma)y \\ \Delta x^\suppred + (1-\gamma)x \\ \Delta \tau^\suppred + (1-\gamma)\tau \end{bmatrix}, G(\Delta z^\suppred + (1-\gamma)z) \right\rangle \\
		&= \langle \Delta x^\suppred + (1-\gamma) x, \Delta s^\suppred + (1-\gamma) s \rangle + (\Delta \tau^\suppred + (1-\gamma)\tau ) (\Delta \kappa^\suppred + (1-\gamma)\kappa).
	\end{align*}
	Therefore,
	\begin{equation}
		\label{eq:PredictorOrthogonalIntermediate}
		\langle \Delta x^\suppred, \Delta s^\suppred \rangle + \Delta \tau^\suppred \Delta \kappa^\suppred = -(1-\gamma)[ \langle x, \Delta s^\suppred \rangle + \langle \Delta x^\suppred, s \rangle + \tau \Delta \kappa^\suppred + \Delta \tau^\suppred \kappa ] - (1-\gamma)^2[\langle x,s \rangle + \tau \kappa].
	\end{equation}
	From \labelcref{eq:AffineDirectionXS,eq:CenteringDirectionXS}, we can see that
	\begin{align*}
		\langle x, \Delta s^\suppred \rangle + \langle \Delta x^\suppred, s \rangle
		&= \langle x, \Delta s^\supaff + \gamma \Delta s^\supcen \rangle + \langle \Delta x^\supaff + \gamma \Delta x^\supcen, W x \rangle\\
		&= \langle x, \Delta s^\supaff + W \Delta x^\supaff \rangle + \gamma \langle x, \Delta s^\supcen + W \Delta x^\supcen \rangle\\
		&= \langle x, -s \rangle + \gamma \langle x, \mue \tilde{s} \rangle = \gamma \mue \nu - \langle x,s \rangle.
	\end{align*}
	Moreover, \labelcref{eq:AffineDirectionTauKappa,eq:CenteringDirectionTauKappa} show that
	\begin{align*}
		\tau \Delta \kappa^\suppred + \Delta \tau^\suppred \kappa 
		&= \tau (\Delta \kappa^\supaff + \gamma \Delta \kappa^\supcen) + (\Delta \tau^\supaff + \gamma \Delta \tau^\supcen) \kappa\\
		&= \tau \Delta \kappa^\supaff + \Delta \tau^\supaff \kappa + \gamma ( \tau \Delta \kappa^\supcen + \Delta \tau^\supcen \kappa )\\
		&= - \kappa \tau + \gamma \mue.
	\end{align*}
	Combining the above, we get
	\begin{equation}
		\label{eq:PredictorComplementarityIntermediate}
		\langle x, \Delta s^\suppred \rangle + \langle \Delta x^\suppred, s \rangle + \tau \Delta \kappa^\suppred + \Delta \tau^\suppred \kappa
		= \gamma \mue \nu - \langle x,s \rangle - \kappa \tau + \gamma \mue = (\gamma - 1)[\langle x,s \rangle + \tau \kappa],
	\end{equation}
	by definition of $\mue$. Hence, \eqref{eq:PredictorOrthogonalIntermediate} must be zero.

	\ref{item:PredictorComplementarity}: Using \ref{item:PredictorOrthogonal} to substitute $\langle \Delta x^\suppred, \Delta s^\suppred \rangle + \Delta \tau^\suppred \Delta \kappa^\suppred$, we get
	\begin{align*}
		&\langle x+\alpha\Delta x^\suppred, s + \alpha \Delta s^\suppred \rangle + (\tau+\alpha\Delta \tau^\suppred)(\kappa + \alpha \Delta \kappa^\suppred)\\
		&= \langle x,s \rangle + \tau \kappa + \alpha[\langle x, \Delta s^\suppred \rangle + \langle \Delta x^\suppred, s \rangle  + \tau \Delta \kappa^\suppred + \Delta \tau^\suppred \kappa] + \alpha^2 0
		= (1-\alpha(1-\gamma))[\langle x,s \rangle + \tau \kappa],
	\end{align*}
	where the final equality is due to \eqref{eq:PredictorComplementarityIntermediate}.
\end{proof}

\subsection{Norm of the Predictor}
\label{subsec:PredictorNorm}
Let us now consider what happens to \ref{ass:XSinCone} after the predictor step.
A question that needs answering is how big the norms of the primal and dual predictors $\Delta x^\suppred$ and $\Delta s^\suppred$ can be. Without an upper bound on the norm of these predictors, it might be that taking the step from $z$ to $z_\subone$ yields an $x_\subone \notin \interior K$ or $s_\subone \notin \interior K^*$. In this section, we will propose an upper bound on the norms of $\Delta x^\suppred$ and $\Delta s^\suppred$.

As we will see, the norms of the predictors depend on $\Delta \tau^\suppred \Delta \kappa^\suppred$. We first show that $\Delta \tau^\suppred \Delta \kappa^\suppred \leq (\gamma \mue - \tau \kappa)^2 / (4 \tau \kappa)$.
\begin{lemma}
	\label{lemma:DeltaTauDeltaKappaUpperBound}
	Let $\tau, \kappa, \mue > 0$ be positive constants, and let $\gamma \geq 0$. Then, the optimization problem
	\begin{align*}
		\max_{\Delta \kappa^\supaff, \Delta \kappa^\supcen, \Delta \tau^\supaff, \Delta \tau^\supcen}\, & (\Delta \kappa^\supaff + \gamma \Delta \kappa^\supcen) (\Delta \tau^\supaff + \gamma \Delta \tau^\supcen)\\
		\text{subject to}\, & \tau \Delta \kappa^\supaff + \kappa \Delta \tau^\supaff = - \tau \kappa\\
		& \tau \Delta \kappa^\supcen + \kappa \Delta \tau^\supcen = \mue,
	\end{align*}
	has optimal value $(\gamma \mue - \tau \kappa)^2 / (4 \tau \kappa)$. 
\end{lemma}
\begin{proof}
	The constraints show that $\Delta \tau^\supaff = -\tau - \frac{\tau}{\kappa} \Delta \kappa^\supaff$ and $\Delta \tau^\supcen = \frac{\mue}{\kappa} - \frac{\tau}{\kappa} \Delta \kappa^\supcen$. Then, the objective is equal to
	\begin{align*}
		(\Delta \kappa^\supaff + \gamma \Delta \kappa^\supcen) (\Delta \tau^\supaff + \gamma \Delta \tau^\supcen) 
		&= (\Delta \kappa^\supaff + \gamma \Delta \kappa^\supcen) \left(-\tau - \frac{\tau}{\kappa} \Delta \kappa^\supaff + \gamma \left( \frac{\mue}{\kappa} - \frac{\tau}{\kappa} \Delta \kappa^\supcen \right) \right)\\
		&= -\frac{\tau}{\kappa} (\Delta \kappa^\supaff + \gamma \Delta \kappa^\supcen)^2 + (\Delta \kappa^\supaff + \gamma \Delta \kappa^\supcen) \left(-\tau + \gamma \frac{\mue}{\kappa} \right).
	\end{align*}
	This expression is maximized if and only if the first order condition
	\begin{equation*}
		-\frac{2\tau}{\kappa} (\Delta \kappa^\supaff + \gamma \Delta \kappa^\supcen) + \left(-\tau + \gamma \frac{\mue}{\kappa} \right) = 0,
	\end{equation*}
	holds. Hence, all maximizers satisfy $\Delta \kappa^\supaff + \gamma \Delta \kappa^\supcen = \frac{1}{2}(\gamma \mue/\tau - \kappa)$. Therefore, the optimal value is
	\begin{align*}
		-\frac{\tau}{\kappa} (\Delta \kappa^\supaff + \gamma \Delta \kappa^\supcen)^2 + (\Delta \kappa^\supaff + \gamma \Delta \kappa^\supcen) \left(-\tau + \gamma \frac{\mue}{\kappa} \right) &= -\frac{\tau}{4\kappa} \left( \frac{\gamma \mue}{\tau} - \kappa \right)^2 + \frac{1}{2}\left( \frac{\gamma \mue}{\tau} - \kappa \right) \left(-\tau + \gamma \frac{\mue}{\kappa} \right)\\
		&= \frac{(\gamma \mue - \tau \kappa)^2}{4 \tau \kappa}. \qedhere
	\end{align*}
\end{proof}

We are now ready to bound the norm of the predictor.

\begin{theorem}
	\label{thm:PredictorNorms}
	Let $K \subset \mathbb{R}^n$ be a proper cone admitting a $\nu$-LHSCB $\primbar$. Pick $z = (y, x, \tau, s, \kappa)$ such that \labelcref{ass:XSinCone,ass:TauKappaBeta,ass:TauKappaPositive,ass:MuTildeBeta,ass:PrimDelta} hold for some $\beta \in (0,1]$ and $\eta \in [0, 1)$. Let $\primlow, \dualup > 0$ be bounds such that \eqref{eq:ScalingHessianAssumption} holds, and let $\gamma \in \mathbb{R}$.
	Then, the primal and dual predictors $\Delta x^\suppred$ and $\Delta s^\suppred$ satisfy
	\begin{equation}
		\label{eq:PredictorNormWBound}
		\| \Delta x^\suppred \|_W^2 + (\| \Delta s^\suppred \|_W^*)^2 \leq  \mue \left[ \nu \left( 1 -2\gamma + \frac{\gamma^2}{\beta} \right) +1 - \tfrac{1}{2} \beta + \frac{\gamma^2}{2 \beta} - \gamma \right], 
	\end{equation}
	and
	\begin{equation} 
		\label{eq:PredictorNormLocalBound}
		\primlow \| \Delta x^\suppred \|^2_x + \frac{1}{\dualup} \| \Delta s^\suppred \|^2_s \leq \frac{1}{\beta} \left[ \nu \left( 1 -2\gamma + \frac{\gamma^2}{\beta} \right) + 1 - \tfrac{1}{2} \beta + \frac{\gamma^2}{2 \beta} - \gamma \right].
	\end{equation}
\end{theorem}
\begin{proof}
	Note that for any two vectors $v, w$ with $\langle v, w\rangle = 0$, we have $\| v + w \|^2 = \| v \|^2 + \| w \|^2$.
	As shown by Lemma \ref{lemma:PredictorProperties}\ref{item:PredictorOrthogonal}, $\langle (W^{1/2} \Delta x^\suppred, \Delta \tau^\suppred), (W^{-1/2} \Delta s^\suppred, \Delta \kappa^\suppred) \rangle = 0$, so
	\begin{equation*}
		(\| W \Delta x^\suppred + \Delta s^\suppred \|_W^*)^2 + (\Delta \tau^\suppred + \Delta \kappa^\suppred)^2 = \| \Delta x^\suppred \|_W^2 + (\| \Delta s^\suppred \|_W^*)^2 + (\Delta \tau^\suppred)^2 + (\Delta \kappa^\suppred)^2.
	\end{equation*}
	By \labelcref{eq:AffineDirectionXS,eq:CenteringDirectionXS,eq:DefPredictor}, $W \Delta x^\suppred + \Delta s^\suppred = - s + \gamma \mue \tilde{s}$. Therefore,
	\begin{align*}
		\| \Delta x^\suppred \|_W^2 + (\| \Delta s^\suppred \|_W^*)^2 
		&= (\| - s + \gamma \mue \tilde{s} \|_W^*)^2 + 2 \Delta \tau^\suppred \Delta \kappa^\suppred \\
		&= \langle x,s \rangle - 2 \gamma \mue \langle x, \tilde{s} \rangle + \gamma^2 (\mue)^2 \langle \tilde{x}, \tilde{s} \rangle + 2 \Delta \tau^\suppred \Delta \kappa^\suppred\\
		&= \langle x,s \rangle - 2 \gamma \mue \nu + \gamma^2 (\mue)^2 \langle \tilde{x}, \tilde{s} \rangle + 2 \Delta \tau^\suppred \Delta \kappa^\suppred,
	\end{align*}
	where the final equality used $\langle x, \tilde{s} \rangle = - \langle x, \primgrad(x) \rangle = \nu$. By Lemma \ref{lemma:DeltaTauDeltaKappaUpperBound},
	\begin{align*}
		\| \Delta x^\suppred \|_W^2 + (\| \Delta s^\suppred \|_W^*)^2 
		&\leq \langle x,s \rangle - 2 \gamma \mue \nu + \gamma^2 (\mue)^2 \langle \tilde{x}, \tilde{s} \rangle + \frac{(\gamma \mue - \tau \kappa)^2}{2 \tau \kappa}\\
		&= \langle x,s \rangle - 2 \gamma \mue \nu + \gamma^2 (\mue)^2 \langle \tilde{x}, \tilde{s} \rangle + \frac{\gamma^2 (\mue)^2}{2 \tau \kappa} - \gamma \mue + \tfrac{1}{2} \tau\kappa.
	\end{align*}
	To bound all the above in terms of $\mue$, we use \labelcref{ass:MuTildeBeta,ass:TauKappaBeta}:
	\begin{align*}
		\| \Delta x^\suppred \|_W^2 + (\| \Delta s^\suppred \|_W^*)^2 
		&\leq (\langle x,s \rangle + \tau \kappa) - \tfrac{1}{2} \tau\kappa - 2 \gamma \mue \nu + \gamma^2 \mue \frac{\nu}{\beta} + \frac{\gamma^2 \mue}{2 \beta} - \gamma \mue\\
		&\leq \mue \left[ (\nu+1) -\tfrac{1}{2} \beta - 2 \gamma \nu + \gamma^2 \frac{\nu}{\beta} + \frac{\gamma^2}{2 \beta} - \gamma \right],
	\end{align*}
	which proves \eqref{eq:PredictorNormWBound}.
	Towards proving \eqref{eq:PredictorNormLocalBound}, we note that by \eqref{eq:ScalingHessianAssumption} and Lemma \ref{lemma:OrderInverseMatrix},
	\begin{equation*}
		\primlow \| \Delta x^\suppred \|^2_x + \frac{1}{\dualup} \| \Delta s^\suppred \|^2_s 
		\leq \frac{\| \Delta x^\suppred \|_W^2 + (\| \Delta s^\suppred \|_W^*)^2 }{\mu}.
	\end{equation*}
	The claim now follows from Lemma \ref{lemma:MuMueBounds} and \eqref{eq:PredictorNormWBound}.
\end{proof}

We remark that both that the coefficient of $\nu$ and the constant term in \labelcref{eq:PredictorNormWBound,eq:PredictorNormLocalBound} are nonnegative for all $\gamma \in \mathbb{R}$ and $\beta \in (0,1]$. Concerning the coefficient, since $\gamma^2 \geq 0$, $\gamma^2/\beta \geq \gamma^2$ for all $\beta \in (0,1]$. Then, 
\begin{equation}
	\label{eq:NuCoefficientNonnegative}
	1 -2\gamma + \gamma^2/\beta \geq 1 - 2 \gamma + \gamma^2 = (\gamma-1)^2 \geq 0.
\end{equation}
Next, we show that the constant terms in \labelcref{eq:PredictorNormWBound,eq:PredictorNormLocalBound} are nonnegative.
To see this, it suffices to show $\beta - \tfrac{1}{2} \beta^2 + \tfrac{1}{2}\gamma^2 - \gamma \beta \geq 0$ for all $\gamma \in \mathbb{R}$ and $\beta \in (0,1]$. Since $\beta - \tfrac{1}{2} \beta^2 + \tfrac{1}{2}\gamma^2 - \gamma \beta$ is concave in $\beta$, the claim must hold if $\beta - \tfrac{1}{2} \beta^2 + \tfrac{1}{2}\gamma^2 - \gamma \beta \geq 0$ for all $\beta \in \{0,1\}$ and $\gamma \in \mathbb{R}$.
For $\beta = 0$, $\beta - \tfrac{1}{2} \beta^2 + \tfrac{1}{2}\gamma^2 - \gamma \beta = \tfrac{1}{2}\gamma^2 \geq 0$ for all $\gamma \in \mathbb{R}$, while for $\beta = 1$, $\beta - \tfrac{1}{2} \beta^2 + \tfrac{1}{2}\gamma^2 - \gamma \beta = \tfrac{1}{2}\gamma^2 - \gamma + \tfrac{1}{2} = \tfrac{1}{2} (\gamma - 1)^2 \geq 0$. Hence,
\begin{equation}
	\label{eq:NuConstantNonnegative}
	1 - \tfrac{1}{2} \beta + \frac{\gamma^2}{2 \beta} - \gamma = \frac{\beta - \tfrac{1}{2} \beta^2 + \tfrac{1}{2}\gamma^2 - \gamma \beta }{\beta} \geq 0.
\end{equation}

Moreover, the bounds \labelcref{eq:PredictorNormWBound,eq:PredictorNormLocalBound} are tight in the simple case where $\beta = \gamma = 1$: one can verify that the bounds are then both equal to zero, indicating that $\Delta x^\suppred$ and $ \Delta s^\suppred $ are both zero.

To compress the notation of the upper bounds \labelcref{eq:PredictorNormWBound,eq:PredictorNormLocalBound} somewhat, we define
\begin{equation}
	\label{eq:DefNormub}
	\normub \coloneqq \left[ \nu \left( 1 -2\gamma + \frac{\gamma^2}{\beta} \right) +1 - \tfrac{1}{2} \beta + \frac{\gamma^2}{2 \beta} - \gamma \right],
\end{equation}
such that $\| \Delta x^\suppred \|_W^2 + (\| \Delta s^\suppred \|_W^*)^2 \leq \mue \normub$ and $\primlow \| \Delta x^\suppred \|^2_x +  \| \Delta s^\suppred \|^2_s / \dualup \leq \normub / \beta$.

The following corollary is a straightforward consequence of Theorem \ref{thm:PredictorNorms}.
\begin{corollary}
	\label{cor:InnerProductPredictorBound}
	Under the conditions of Theorem \ref{thm:PredictorNorms},
	\begin{equation*}
		| \langle \Delta x^\suppred, \Delta s^\suppred \rangle | \leq \tfrac{1}{2} \mue \normub \leq \frac{\normub \mu}{2 \beta} .
	\end{equation*}
\end{corollary}
\begin{proof}
	Note that
	\begin{equation*}
		0 \leq (\| W \Delta x^\suppred + \Delta s^\suppred \|_W^*)^2 = \| \Delta x^\suppred \|_W^2 + (\| \Delta s^\suppred \|_W^*)^2 + 2 \langle \Delta x^\suppred, \Delta s^\suppred \rangle,
	\end{equation*}
	which implies $-\langle \Delta x^\suppred, \Delta s^\suppred \rangle \leq \frac{1}{2} [\| \Delta x^\suppred \|_W^2 + (\| \Delta s^\suppred \|_W^*)^2]$. The upper bound then follows from \eqref{eq:PredictorNormWBound}. Similarly, $\langle \Delta x^\suppred, \Delta s^\suppred \rangle \leq \frac{1}{2} [\| \Delta x^\suppred \|_W^2 + (\| \Delta s^\suppred \|_W^*)^2]$ by considering $(\| W \Delta x^\suppred - \Delta s^\suppred \|_W^*)^2$. The final inequality is due to Lemma \ref{lemma:MuMueBounds}.
\end{proof}
As a second consequence, we can bound the difference between $\mu_\subone$ and $\mu$, in the following sense.
\begin{corollary}
	\label{cor:MuOneOverMu}
	Under the conditions of Theorem \ref{thm:PredictorNorms},
	\begin{equation*}
		\frac{\mu_\subone}{\mu} 
		= 1 + \frac{\alpha^2}{\mu \nu} \langle \Delta x^\suppred, \Delta s^\suppred \rangle + \alpha \left(  \gamma \frac{\mue}{\mu} - 1 \right)
		\leq 1 + \frac{\alpha^2 \normub}{2 \beta \nu} + \alpha \left(  \frac{\gamma }{\beta} - 1 \right).
	\end{equation*}
\end{corollary}
\begin{proof}
	By definition \eqref{eq:DefPredictor},
	\begin{equation*}
		\frac{\mu_\subone}{\mu}
		= \frac{\langle x + \alpha \Delta x^\suppred, s + \alpha \Delta s^\suppred \rangle}{\mu \nu}
		= 1 + \frac{\alpha}{\mu \nu} \langle x, W \Delta x^\suppred + \Delta s^\suppred \rangle + \frac{\alpha^2}{\mu \nu} \langle \Delta x^\suppred, \Delta s^\suppred \rangle.
	\end{equation*}
	It follows from \eqref{eq:AffineDirectionXS}, \eqref{eq:CenteringDirectionXS}, and \eqref{eq:DefPredictor} that $W \Delta x^\suppred + \Delta s^\suppred = - s + \gamma \mue \tilde{s}$. This proves the desired equality. The inequality follows from Corollary \ref{cor:InnerProductPredictorBound} and Lemma \ref{lemma:MuMueBounds}.
\end{proof}

\subsection{Tau Kappa After Prediction}
\label{subsec:PredictorTauKappa}
With \Cref{thm:PredictorNorms}, we can derive values for $\alpha$ such that $x_\subone$ and $s_\subone$ lie in the Dikin ellipsoid around $x$ and $s$ respectively, which implies $x_\subone \in \interior K$ and $s_\subone \in \interior K^*$. We therefore move on to the question what happens to \labelcref{ass:TauKappaPositive,ass:TauKappaBeta} after the predictor step.
This question is answered by the following lemma.


\begin{lemma}
	\label{lemma:PredictorTauKappa}
	Let $K \subset \mathbb{R}^n$ be a proper cone admitting a $\nu$-LHSCB $\primbar$. Pick $z = (y, x, \tau, s, \kappa)$ such that \labelcref{ass:XSinCone,ass:TauKappaBeta,ass:TauKappaPositive,ass:MuTildeBeta,ass:PrimDelta} hold for some $\beta \in (0,1]$ and $\eta \in [0, 1)$. Let $\alpha \leq 1$ and $\gamma \in \mathbb{R}$.
	Then, 
	\begin{equation}
		\label{eq:TauKappaOneLowerBound}
		\tau_\subone \kappa_\subone 
		\geq \frac{\mue_\subone}{1 - \alpha(1-\gamma)} \left( \beta (1-\alpha) + \alpha \gamma - \tfrac{1}{2} \alpha^2 \normub \right).
	\end{equation}
	If $\beta < 1$ or $\gamma < 1$, and 
	\begin{equation}
		\label{eq:TauKappaPositiveCondition}
		0 \leq \alpha < \frac{\gamma - \beta + \sqrt{(\beta-\gamma)^2 + 2 \beta \normub}}{\normub},
	\end{equation}
	then $\tau_\subone > 0$ and $\kappa_\subone >0$.
\end{lemma}
\begin{proof}
	Recall that $z_\subone = z + \alpha \Delta z^\suppred$.
	By \labelcref{eq:AffineDirectionTauKappa,eq:CenteringDirectionTauKappa,eq:DefPredictor},
	\begin{align*}
		\tau_\subone \kappa_\subone 
		&= \tau \kappa + \alpha (\tau \Delta \kappa^\suppred + \kappa \Delta \tau^\suppred) + \alpha^2 \Delta \tau^\suppred \Delta \kappa^\suppred\\
		&= \tau \kappa + \alpha (-\tau \kappa + \gamma \mue) + \alpha^2 \Delta \tau^\suppred \Delta \kappa^\suppred\\
		&\geq (1- \alpha) \beta \mue + \alpha \gamma \mue + \alpha^2 \Delta \tau^\suppred \Delta \kappa^\suppred,
	\end{align*}
	where the inequality is due to \ref{ass:TauKappaBeta}. Hence, the main remaining task is to lower bound $\Delta \tau^\suppred \Delta \kappa^\suppred$. By Lemma \ref{lemma:PredictorProperties}\ref{item:PredictorOrthogonal}, $\Delta \tau^\suppred \Delta \kappa^\suppred = -\langle \Delta x^\suppred, \Delta s^\suppred \rangle$. 
	It therefore follows from Corollary \ref{cor:InnerProductPredictorBound} that $ \tau_\subone \kappa_\subone \geq \mue \left( \beta (1-\alpha) + \alpha \gamma - \tfrac{1}{2} \alpha^2 \normub \right). $
	Lemma \ref{lemma:PredictorProperties}\ref{item:PredictorComplementarity} then 
	proves the lower bound.
	
	It can be seen from \eqref{eq:NuCoefficientNonnegative} that $1 - 2\gamma + \gamma^2/\beta \leq 0$ only if $\beta = \gamma = 1$. 
	To avoid trivial difficulties, we therefore assume that $\beta < 1$ or $\gamma < 1$. Then, $\normub = \nu \left( 1 -2\gamma + \gamma^2 / \beta \right) +1 - \tfrac{1}{2} \beta + \tfrac{1}{2} \gamma^2 /\beta - \gamma$ is positive by \eqref{eq:NuConstantNonnegative}.
	The zeros of the right hand side of \eqref{eq:TauKappaOneLowerBound} in $\alpha$ are
	\begin{equation*}
		\frac{\gamma - \beta \pm \sqrt{(\beta-\gamma)^2 + 2 \beta \normub}}{\normub}.
	\end{equation*}
	We see that the square root above is greater than $|\gamma - \beta|$, and the denominator is positive. Thus, for all $\alpha$ satisfying \eqref{eq:TauKappaPositiveCondition}, the lower bound \eqref{eq:TauKappaOneLowerBound} on $\tau_\subone \kappa_\subone$ is positive.
	
	If $\tau_\subone \kappa_\subone > 0$, then $\tau_\subone > 0$ and $\kappa_\subone > 0$, or $\tau_\subone < 0$ and $\kappa_\subone < 0$. For the sake of contradiction, suppose that $\tau_\subone < 0$ and $\kappa_\subone < 0$. Since $\alpha \mapsto \tau + \alpha \Delta \tau^\suppred $ is continuous and $\tau > 0$, the intermediate value theorem implies that there exists some $\alpha_0 \in (0, \alpha)$ where $\tau + \alpha_0 \Delta \tau^\suppred = 0$. But for this $\alpha_0$, we have $(\tau + \alpha_0 \Delta \tau^\suppred) (\kappa + \alpha_0 \Delta \kappa^\suppred) = 0$, while the lower bound \eqref{eq:TauKappaOneLowerBound} is positive for all $\alpha_0 \in (0, \alpha)$. Thus, we have a contradiction, and therefore $\tau_\subone > 0$ and $\kappa_\subone > 0$.
\end{proof}

\subsection{Shadow Distance After Prediction}
\label{subsec:PredictorShadowDistance}
One may wonder what happens to assumption \ref{ass:MuTildeBeta} after the predictor step. However, we did not find a satisfying way to bound $\mue_\subone \tilde{\mu}_\subone$. This is the main reason we introduced a corrector in this analysis, which we consider in more detail in Section \ref{sec:Corrector}.

To complete the analysis of the predictor, we consider \ref{ass:PrimDelta}. Its analysis will be simplified by the observation that if $\lambda \in (0,1]$ and $\epsilon \geq 0$, then
\begin{equation}
	\label{eq:MaximumSCAidOne}
	\frac{1+\epsilon}{\lambda} - 1 \geq 1 - \lambda (1-\epsilon).
\end{equation}
\begin{lemma}
	\label{lemma:PrimDeltaOne}
	Let $K \subset \mathbb{R}^n$ be a proper cone admitting a $\nu$-LHSCB $\primbar$. Pick $z = (y, x, \tau, s, \kappa)$ such that \labelcref{ass:XSinCone,ass:TauKappaBeta,ass:TauKappaPositive,ass:MuTildeBeta,ass:PrimDelta} hold for some $\beta \in (0,1]$ and $\eta \in [0, 1)$. 
	Assume \eqref{eq:ScalingHessianAssumption} holds with $\dualup = (1+\epsilon) / \lambda$ and $\duallow = \lambda (1-\epsilon)$
	for some $\lambda \in (0,1]$ and $\epsilon \geq 0$. Let $\alpha \in [0,\sqrt{\beta  \min\{ \primlow, 1/\primup \} / \theta })$ and $\gamma \in \mathbb{R}$.
	Then, $\| \primdelta_\subone \|_{x_\subone}$ is at most
	\begin{equation*} 
		\frac{1}{1 - \alpha \sqrt{\normub / (\beta \primlow)}} \left( (1-\alpha) \|\primdelta \|_x + \frac{\alpha^2 \normub}{2 \beta \sqrt{\nu} (1- \| \primdelta \|_x)}
		+ \frac{\alpha \sqrt{\normub}}{\sqrt{\primlow \beta}} \left( \omega_1 \left( \frac{\dualup}{1 - \alpha \sqrt{\dualup \normub/\beta} } - 1 \right) + \omega_2 \right) \right),
	\end{equation*}
	where
	\begin{equation*}
		\omega_1 \coloneqq 1 + \frac{\alpha^2 \normub}{2 \beta \nu} + \alpha \left(  \frac{\gamma }{\beta} - 1 \right) \qquad \text{and} \qquad \omega_2 \coloneqq \frac{\alpha^2 \normub}{2 \beta \nu} + \alpha \max\left\{ \frac{\gamma }{\beta} - 1, 1 - \frac{\gamma}{2-\beta} \right\}.
	\end{equation*}
\end{lemma}
\begin{proof}
	By \labelcref{eq:AffineDirectionXS,eq:CenteringDirectionXS,eq:DefPredictor}, and the fundamental theorem of calculus,
	\begin{align}
		\primdelta_\subone &= x_\subone - \mu_\subone \tilde{x}_\subone \nonumber\\
		&= x + \alpha \Delta x^\suppred + \mu_\subone \dualgrad \left( s + \alpha \Delta s^\suppred \right) \nonumber\\
		&= x + \alpha \left( \gamma \mue \tilde{x} - x - W^{-1} \Delta s^\suppred \right) + \mu_\subone \left[ \dualgrad( s) + \int_0^1 \dualhess\left( s + t \alpha \Delta s^\suppred \right) \alpha \Delta s^\suppred \diff t \right] \nonumber\\
		&= ( 1 - \alpha) x + \alpha \left( \gamma \mue \tilde{x} - W^{-1} \Delta s^\suppred \right) - \mu_\subone \tilde{x} + \alpha \int_0^1 \mu_\subone \dualhess\left( s + t \alpha \Delta s^\suppred \right) \Delta s^\suppred \diff t \nonumber\\
		&= ( 1 - \alpha ) (x - \mu \tilde{x}) 
		+ \left(\mu - \mu_\subone + \alpha \left( \gamma \mue - \mu \right) \right) \tilde{x} + \alpha \int_0^1 \left[ \mu_\subone \dualhess\left( s + t \alpha \Delta s^\suppred \right) - W^{-1} \right] \Delta s^\suppred \diff t \nonumber\\
		&= \left( 1 - \alpha \right) \primdelta
		+ \left(1 - \frac{\mu_\subone}{\mu} + \alpha \left( \gamma \frac{\mue}{\mu} - 1 \right) \right) \mu \tilde{x} + \alpha \int_0^1 \left[ \mu_\subone \dualhess\left( s + t \alpha \Delta s^\suppred \right) - W^{-1} \right] \Delta s^\suppred \diff t. \label{eq:PrimDeltaOneThreeTerms}
	\end{align}
	We will bound each of the three terms in \eqref{eq:PrimDeltaOneThreeTerms} separately in the $x$-norm. Of course, $\| \primdelta \|_x$ is known to be bounded by assumption \ref{ass:PrimDelta}. 
	
	To bound the second term of \eqref{eq:PrimDeltaOneThreeTerms}, we note that by Corollaries \ref{cor:MuOneOverMu} and \ref{cor:InnerProductPredictorBound},
	\begin{equation*}
		\left| 1-\frac{\mu_\subone}{\mu} + \alpha \left(  \gamma \frac{\mue}{\mu} - 1 \right) \right| = \left| - \frac{\alpha^2}{\mu \nu} \langle  \Delta x^\suppred, \Delta s^\suppred \rangle \right|
		\leq \frac{\alpha^2 \normub}{2 \beta \nu}.
	\end{equation*}
	Therefore, since $\| \mu \tilde{x} \|_x \leq \| \mu \tilde{x} \|_{\mu \tilde{x}} / (1- \| \primdelta \|_x) = \sqrt{\nu} / (1- \| \primdelta \|_x)$ by \eqref{eq:DefSelfConcordance},
	\begin{equation}
		\label{eq:PrimDeltaOneSecondTerm}
		\left\| \left(1 - \frac{\mu_\subone}{\mu} + \alpha \left(  \gamma \frac{\mue}{\mu} - 1 \right) \right) \mu \tilde{x} \right\|_{x} 
		\leq \frac{\alpha^2 \normub}{2 \beta \sqrt{\nu} (1- \| \primdelta \|_x)}.
	\end{equation}
	
	To bound the last term in \eqref{eq:PrimDeltaOneThreeTerms}, note that $\left\| \left[ \mu_\subone \dualhess\left( s + t \alpha \Delta s^\suppred \right) - W^{-1} \right] \Delta s^\suppred \right\|_W$ is at most
	\begin{equation}
		\label{eq:PrimDeltaOneTermThree}
		\frac{\mu_\subone}{\mu} \left\| \left[ \mu \dualhess\left( s + t \alpha \Delta s^\suppred \right) - W^{-1} \right] \Delta s^\suppred \right\|_W + \left| 1- \frac{\mu_\subone}{\mu}  \right| \| \Delta s^\suppred \|_W^*.
	\end{equation}
	The first term in \eqref{eq:PrimDeltaOneTermThree} can be bounded using the operator norm \eqref{eq:DefOperatorNorm}:
	\begin{align}
		\left\| \left[ \mu \dualhess\left( s + t \alpha \Delta s^\suppred \right) - W^{-1} \right] \Delta s^\suppred \right\|_W 
		&\leq \| \Delta s^\suppred \|_W^* \sup_{\| u \|_W^* \leq 1} \left\| \left[ \mu \dualhess\left( s + t \alpha \Delta s^\suppred \right) - W^{-1} \right] u \right\|_W \nonumber\\
		&= \| \Delta s^\suppred \|_W^* \left\| \mu \dualhess\left( s + t \alpha \Delta s^\suppred \right) - W^{-1} \right\|_{W^{-1}}. \label{eq:PrimDeltaOneTermThreeIntermediateOne}
	\end{align}
	With Lemma \ref{lemma:NormToMatrixInequality} in mind, we will proceed by bounding $\mu \dualhess( s + t \alpha \Delta s^\suppred )$ in terms of $W^{-1}$. By self-concordance, \eqref{eq:ScalingHessianAssumption}, and Lemma \ref{lemma:OrderInverseMatrix},
	\begin{equation*}
		 \mu \dualhess\left( s + t \alpha \Delta s^\suppred \right) \preceq \frac{1}{(1 - t \alpha \| \Delta s^\suppred \|_s )^2} \mu \dualhess( s ) \preceq \frac{\dualup}{(1 - t \alpha \| \Delta s^\suppred \|_s )^2} W^{-1}.
	\end{equation*}
	Similarly,
	\begin{equation*}
		\mu \dualhess\left( s + t \alpha \Delta s^\suppred \right) \succeq (1 - t \alpha \| \Delta s^\suppred \|_s )^2 \mu \dualhess( s ) \succeq (1 - t \alpha \| \Delta s^\suppred \|_s )^2 \duallow W^{-1}.
	\end{equation*}
	Therefore, Lemma \ref{lemma:NormToMatrixInequality} shows
	\begin{align}
		\left\| \mu \dualhess\left( s + t \alpha \Delta s^\suppred \right) - W^{-1} \right\|_{W^{-1}} 
		&\leq \max\left\{ \frac{\dualup}{(1 - t \alpha \| \Delta s^\suppred \|_s )^2} - 1, 1 - (1 - t \alpha \| \Delta s^\suppred \|_s )^2 \duallow \right\} \nonumber\\
		&= \frac{\dualup}{(1 - t \alpha \| \Delta s^\suppred \|_s )^2} - 1, \label{eq:PrimDeltaOneTermThreeIntermediateTwo}
	\end{align}
	where the equality holds because of the assumed form of $\dualup$ and $\duallow$, and \eqref{eq:MaximumSCAidOne}. Thus, we have bounded the first term of \eqref{eq:PrimDeltaOneTermThree}.

	To bound $\mu_\subone / \mu$ in \eqref{eq:PrimDeltaOneTermThree}, we can use 
	Corollary \ref{cor:MuOneOverMu}:
	\begin{equation}
		\label{eq:PrimDeltaMuPlusOverMu}
		\frac{\mu_\subone}{\mu} 
		\leq 1 + \frac{\alpha^2 \normub}{2 \beta \nu} + \alpha \left(  \frac{\gamma }{\beta} - 1 \right) = \omega_1.
	\end{equation}
	A lower bound can be found similarly: by Corollary \ref{cor:InnerProductPredictorBound} and Lemma \ref{lemma:MuMueBounds},
	\begin{equation*}
		\frac{\mu_\subone}{\mu} 
		= 1 + \frac{\alpha^2}{\mu \nu} \langle \Delta x^\suppred, \Delta s^\suppred \rangle + \alpha \left(  \gamma \frac{\mue}{\mu} - 1 \right)
		\geq 1 - \frac{\alpha^2 \normub}{2 \beta \nu} + \alpha \left(  \frac{\gamma }{2-\beta} - 1 \right).
	\end{equation*}
	Therefore,
	\begin{equation*}
		\left| 1 - \frac{\mu_\subone}{\mu} \right| \leq \max\left\{ \frac{\alpha^2 \normub}{2 \beta \nu} + \alpha \left(  \frac{\gamma }{\beta} - 1 \right), \frac{\alpha^2 \normub}{2 \beta \nu} - \alpha \left(  \frac{\gamma }{2-\beta} - 1 \right) \right\} 
		= \frac{\alpha^2 \normub}{2 \beta \nu} + \alpha \max\left\{ \frac{\gamma }{\beta} - 1, 1 - \frac{\gamma}{2-\beta} \right\} = \omega_2.
	\end{equation*}
	Combining this bound on $| 1 - \mu_\subone / \mu|$ with \eqref{eq:PrimDeltaOneTermThreeIntermediateOne}, \eqref{eq:PrimDeltaOneTermThreeIntermediateTwo}, and \eqref{eq:PrimDeltaMuPlusOverMu},
	\begin{equation*}
		\left\| \left[ \mu_\subone \dualhess\left( s + t \alpha \Delta s^\suppred \right) - W^{-1} \right] \Delta s^\suppred \right\|_W
		\leq \left( \omega_1 \left( \frac{\dualup}{(1 - t \alpha \| \Delta s^\suppred \|_s )^2} - 1 \right) + \omega_2 \right) \| \Delta s^\suppred \|_W^*.
	\end{equation*}
	
	Consequently, by \eqref{eq:ScalingHessianAssumption},
	\begin{align}
		&\left\| \alpha \int_0^1 \left[ \mu_\subone \dualhess\left( s + t \alpha \Delta s^\suppred \right) - W^{-1} \right] \Delta s^\suppred \diff t \right\|_{x} \nonumber\\
		&\leq \alpha \int_0^1 \frac{1}{\sqrt{\primlow \mu}} \left\| \left[ \mu_\subone \dualhess\left( s + t \alpha \Delta s^\suppred \right) - W^{-1} \right] \Delta s^\suppred \right\|_W \diff t \nonumber\\
		&\leq \frac{\alpha}{\sqrt{\primlow \mu}} \int_0^1 \left( \omega_1 \left( \frac{\dualup}{(1 - t \alpha \| \Delta s^\suppred \|_s )^2} - 1 \right) + \omega_2 \right) \| \Delta s^\suppred \|_W^* \diff t \nonumber\\
		&= \frac{\alpha}{\sqrt{\primlow \mu}} \left( \omega_1 \left( \frac{\dualup}{1 - \alpha \| \Delta s^\suppred \|_s } - 1 \right) + \omega_2 \right) \| \Delta s^\suppred \|_W^* \nonumber \\
		&\leq \frac{\alpha}{\sqrt{\primlow}} \left( \omega_1 \left( \frac{\dualup}{1 - \alpha \sqrt{\dualup \normub/\beta} } - 1 \right) + \omega_2 \right) \sqrt{\normub / \beta}, \label{eq:PrimDeltaOneThirdTerm}
	\end{align}
	where the final inequality uses Theorem \ref{thm:PredictorNorms} and Lemma \ref{lemma:MuMueBounds}.
	
	To complete the proof, we note that by \eqref{eq:DefSelfConcordance} and \eqref{eq:PredictorNormLocalBound}, 
	\begin{equation*}
		\| \primdelta_\subone \|_{x_\subone} 
		\leq \frac{\| \primdelta_\subone \|_{x}}{1 - \alpha \| \Delta x^\suppred \|_{x}} 
		\leq \frac{\| \primdelta_\subone \|_{x}}{1 - \alpha \sqrt{\normub / (\beta \primlow)}}.
	\end{equation*}
	By using the triangle inequality on \eqref{eq:PrimDeltaOneThreeTerms} along with the bounds \eqref{eq:PrimDeltaOneSecondTerm} and \eqref{eq:PrimDeltaOneThirdTerm} on the second and third term of \eqref{eq:PrimDeltaOneThreeTerms}, the claim follows.
\end{proof}

\section{Properties of the Corrector}
\label{sec:Corrector}
We have analyzed what happens to the assumptions \labelcref{ass:XSinCone,ass:TauKappaPositive,ass:TauKappaBeta,ass:PrimDelta} after the predictor step, but we have not investigated \ref{ass:MuTildeBeta}. Moreover, it is not clear from Lemma \ref{lemma:PrimDeltaOne} that we can pick values for $\alpha$, $\beta$, $\gamma$, and $\eta$ such that $\| \primdelta_\subone \|_{x_\subone} \leq \| \primdelta \|_x$. As we will see, the corrector introduced in this work fixes both these problems.

\subsection{Simple Properties}
We start with some simple properties of the corrector \labelcref{eq:CorrectorDirectionG,eq:CorrectorDirectionTauKappa,eq:CorrectorDirectionXS}.

\begin{lemma}
	\label{lemma:CorrectorProperties}
	Let $K \subset \mathbb{R}^n$ be a proper cone admitting a $\nu$-LHSCB $\primbar$. Pick $z_\subone = (y_\subone, x_\subone, \tau_\subone, s_\subone, \kappa_\subone)$ such that $x_\subone \in \interior K$, $s_\subone \in \interior K^*$, and $\tau_\subone, \kappa_\subone > 0$.
	Then,
	\begin{enumerate}[label=(\roman*)]
		\item \label{item:CorrectorG} $G(z_\subone + \Delta z^\supcor_\subone) = G(z_\subone)$
		\item \label{item:CorrectorOrthogonalXS} $\langle x_\subone, \Delta s^\supcor_\subone \rangle + \langle \Delta x^\supcor_\subone, s_\subone \rangle  = 0$
		\item \label{item:CorrectorOrthogonal} $\langle \Delta x^\supcor_\subone, \Delta s^\supcor_\subone \rangle + \Delta \tau^\supcor_\subone \Delta \kappa^\supcor_\subone = 0$
		\item \label{item:CorrectorComplementarity} $\langle x_\subone + \Delta x^\supcor_\subone, s_\subone + \Delta s^\supcor_\subone \rangle + (\tau_\subone + \Delta \tau^\supcor_\subone) (\kappa_\subone + \Delta \kappa^\supcor_\subone) = \langle x_\subone,s_\subone \rangle + \tau_\subone \kappa_\subone$. In other words, $\mue_\subtwo = \mue_\subone$
		\item \label{item:CorrectorNorm} $\| \Delta x^\supcor_\subone \|_{W_\subone}^2  + (\| \Delta s^\supcor_\subone \|_{W_\subone}^*)^2 \leq \| \primdelta_\subone \|_{W_\subone}^2$
		\item \label{item:CorrectorXSInnerProduct} $| \langle \Delta x^\supcor_\subone, \Delta s^\supcor_\subone \rangle| \leq \frac{1}{2} \| \primdelta_\subone \|_{W_\subone}^2$
		\item \label{item:CorrectorTauKappa} $(\tau_\subone + \Delta \tau^\supcor_\subone) (\kappa_\subone + \Delta \kappa^\supcor_\subone) \geq \tau_\subone \kappa_\subone - \tfrac{1}{2} \| \primdelta_\subone \|_{W_\subone}^2$. If $\tau_\subone \kappa_\subone - \tfrac{1}{2} \| \primdelta_\subone \|_{W_\subone}^2 > 0$, then $\tau_\subtwo, \kappa_\subtwo > 0$.
	\end{enumerate}
\end{lemma}
\begin{proof}
	\ref{item:CorrectorG}: Follows directly from \eqref{eq:CorrectorDirectionG} and the fact that $G$ is a linear operator.
	
	\ref{item:CorrectorOrthogonalXS}: By \eqref{eq:CorrectorDirectionXS} and the fact that $W_\subone x_\subone = s_\subone$,
	\begin{equation*}
		\langle x_\subone, \Delta s^\supcor_\subone \rangle + \langle \Delta x^\supcor_\subone, s_\subone \rangle 
		= \langle x_\subone, \mu_\subone \tilde{s}_\subone - s_\subone - W_\subone \Delta x^\supcor_\subone \rangle + \langle \Delta x^\supcor_\subone, s_\subone \rangle
		= \langle x_\subone, \mu_\subone \tilde{s}_\subone - s_\subone \rangle,
	\end{equation*}
	which is zero because $\langle x_\subone, \tilde{s}_\subone \rangle = \nu$.
	
	\ref{item:CorrectorOrthogonal}: From \eqref{eq:CorrectorDirectionG} and skew-symmetry, it can be seen that
	\begin{equation*}
		0 = \left\langle \begin{bmatrix} \Delta y^\supcor_\subone \\ \Delta x^\supcor_\subone \\ \Delta \tau^\supcor_\subone \end{bmatrix}, G( \Delta z^\supcor_\subone) \right\rangle = \langle \Delta x^\supcor_\subone, \Delta s^\supcor_\subone \rangle + \Delta \tau^\supcor_\subone \Delta \kappa^\supcor_\subone.
	\end{equation*}
	
	\ref{item:CorrectorComplementarity}: A simple expansion shows $\langle x_\subone + \Delta x^\supcor_\subone, s_\subone + \Delta s^\supcor_\subone \rangle + (\tau_\subone + \Delta \tau^\supcor_\subone) (\kappa_\subone + \Delta \kappa^\supcor_\subone)$ equals
	\begin{equation*}
		\langle x_\subone , s_\subone \rangle + \tau_\subone \kappa_\subone + \underbrace{ \langle x_\subone, \Delta s^\supcor_\subone \rangle + \langle \Delta x^\supcor_\subone, s_\subone \rangle}_{= 0 \text{ by \ref{item:CorrectorOrthogonalXS}}} + \underbrace{\tau_\subone \Delta \kappa^\supcor_\subone + \kappa_\subone \Delta \tau^\supcor_\subone}_{=0 \text{ by \eqref{eq:CorrectorDirectionTauKappa}}} + \underbrace{ \langle \Delta x^\supcor_\subone, \Delta s^\supcor_\subone \rangle + \Delta \tau^\supcor_\subone \Delta \kappa^\supcor_\subone}_{=0 \text{ by \ref{item:CorrectorOrthogonal}}}.
	\end{equation*}
	
	\ref{item:CorrectorNorm}: Recall that for any vectors $v$ and $w$ such that $\langle v,w \rangle = 0$, we have $\| v + w \|^2 = \| v \|^2 + \| w\|^2$. By \ref{item:CorrectorOrthogonal}, we have 
	$\langle (W_\subone^{1/2} \Delta x^\supcor_\subone, \Delta \tau^\supcor_\subone ), (W_\subone^{-1/2} \Delta s^\supcor_\subone, \Delta \kappa^\supcor_\subone ) \rangle = 0$, so
	\begin{equation*}
		(\| W_\subone \Delta x^\supcor_\subone + \Delta s^\supcor_\subone \|_{W_\subone}^*)^2 + (\Delta \tau^\supcor_\subone + \Delta \kappa^\supcor_\subone)^2
		= \| \Delta x^\supcor_\subone \|_{W_\subone}^2 + (\| \Delta s^\supcor_\subone \|_{W_\subone}^*)^2 + (\Delta \tau^\supcor_\subone)^2 + (\Delta \kappa^\supcor_\subone)^2.
	\end{equation*}
	Subtracting $(\Delta \tau^\supcor_\subone)^2 + (\Delta \kappa^\supcor_\subone)^2$ from both sides and using $W_\subone \Delta x^\supcor_\subone + \Delta s^\supcor_\subone = - \dualdelta_\subone$, we get
	\begin{equation}
		\label{eq:CorrectorNormIntermediate}
		(\| \dualdelta_\subone \|_{W_\subone}^*)^2 + 2 \Delta \tau^\supcor_\subone \Delta \kappa^\supcor_\subone = \| \Delta x^\supcor_\subone \|_{W_\subone}^2 + (\| \Delta s^\supcor_\subone \|_{W_\subone}^*)^2.
	\end{equation}
	Since $\tau_\subone \Delta \kappa^\supcor_\subone + \kappa_\subone \Delta \tau^\supcor_\subone = 0$ by \eqref{eq:CorrectorDirectionTauKappa}, we can derive the upper bound
	\begin{equation*}
	\Delta \kappa^\supcor_\subone \Delta \tau^\supcor_\subone = \Delta \tau^\supcor_\subone \frac{-\kappa_\subone \Delta \tau^\supcor_\subone}{\tau_\subone} \leq 0.
	\end{equation*}
	Hence, \ref{item:CorrectorNorm} follows from \eqref{eq:CorrectorNormIntermediate} and the observation $(\| \dualdelta_\subone \|_{W_\subone}^*)^2 = \| \primdelta_\subone \|_{W_\subone}^2$.
	
	\ref{item:CorrectorXSInnerProduct}: Can be proven similar to Corollary \ref{cor:InnerProductPredictorBound} using \ref{item:CorrectorNorm}.
	
	\ref{item:CorrectorTauKappa}: From \eqref{eq:CorrectorDirectionTauKappa} and \ref{item:CorrectorOrthogonal}, we get 
	\begin{equation*}
		(\tau_\subone + \Delta \tau^\supcor_\subone) (\kappa_\subone + \Delta \kappa^\supcor_\subone) = \tau_\subone \kappa_\subone + \Delta \tau^\supcor_\subone \Delta \kappa^\supcor_\subone = \tau_\subone \kappa_\subone - \langle \Delta x^\supcor_\subone, \Delta s^\supcor_\subone \rangle.
	\end{equation*}
	The lower bound now follows from \ref{item:CorrectorXSInnerProduct}. If this lower bound on $(\tau_\subone + \Delta \tau^\supcor_\subone) (\kappa_\subone + \Delta \kappa^\supcor_\subone)$ is positive, then either $\tau_\subtwo, \kappa_\subtwo > 0$ or $\tau_\subtwo, \kappa_\subtwo < 0$. Suppose for the sake of contradiction that $\tau_\subtwo, \kappa_\subtwo < 0$. Then, both $\Delta \tau^\supcor_\subone < 0$ and $\Delta \kappa^\supcor_\subone < 0$. However, it follows from \eqref{eq:CorrectorDirectionTauKappa} that these cannot hold at the same time if $\tau_\subone, \kappa_\subone > 0$. Hence, $\tau_\subtwo, \kappa_\subtwo > 0$.
\end{proof}

\subsection{Shadow Inner Product After Correction}
Recall that we did not analyze assumption \ref{ass:MuTildeBeta} in Section \ref{sec:Predictor}. The following result will show that the corrector step not only leads to a small $\| \primdelta_\subtwo \|_{x_\subtwo}$, but also to a small $\mu_\subone \tilde{\mu}_\subtwo$. In Section \ref{sec:Analysis}, we will show that a small $\mu_\subone \tilde{\mu}_\subtwo$ implies a small value of $\mue_\subtwo \tilde{\mu}_\subtwo$.

Similar to \eqref{eq:MaximumSCAidOne}, the following analysis is simplified by noting that if $\lambda \in (0,1]$ and $\epsilon \geq 0$, then
\begin{equation}
	\label{eq:MaximumSCAidTwo}
	\frac{1}{(1-\epsilon)\lambda} - 1 \geq 1 - \frac{\lambda}{1 + \epsilon}.
\end{equation}
We suggest the reader to skip the proof on a first pass.

\begin{lemma}
	\label{lemma:ShadowInnerProduct}
	Let $K \subset \mathbb{R}^n$ be a proper cone admitting a $\nu$-LHSCB $\primbar$. Pick $z_\subone = (y_\subone, x_\subone, \tau_\subone, s_\subone, \kappa_\subone)$ such that $x_\subone \in \interior K$, $s_\subone \in \interior K^*$, and $\tau_\subone, \kappa_\subone > 0$.
	Assume \eqref{eq:ScalingHessianAssumptionOne} holds with $\dualup_\subone = (1+\epsilon) / \lambda$ and $\duallow_\subone = \lambda (1-\epsilon)$ for some $\lambda \in (0,1]$ and $\epsilon \geq 0$. Suppose $\| \primdelta_\subone \|_{W_\subone} / \sqrt{\mu_\subone} < \sqrt{ \min\{\duallow_\subone, 1/\dualup_\subone \} }$.
	Then, $\mu_\subone \tilde{\mu}_\subtwo$ is at most
	\begin{align*}
		&1 + \| \primdelta_\subone \|_{W_\subone} \left[ \frac{\| \tilde{x}_\subone \|_{W_\subone}}{\nu} \left( \frac{1}{\duallow_\subone (1 - \| \primdelta_\subone \|_{W_\subone} / \sqrt{ \duallow_\subone \mu_\subone})} - 1 \right) 
		+ \frac{\dualup_\subone }{\sqrt{\mu_\subone \nu} } \left( \frac{1}{ 1 - \sqrt{\dualup_\subone / \mu_\subone} \| \primdelta_\subone \|_{W_\subone}} - 1 \right) \right]\\
		& + \frac{\| \primdelta_\subone \|_{W_\subone}^2 \dualup_\subone}{\mu_\subone \nu \duallow_\subone (1 - \| \primdelta_\subone \|_{W_\subone} / \sqrt{ \duallow_\subone \mu_\subone}) (1 - \sqrt{\dualup_\subone / \mu_\subone} \| \primdelta_\subone \|_{W_\subone})}.
	\end{align*}
\end{lemma}
\begin{proof}
	It follows from \eqref{eq:CorrectorDirectionXS} that $x_\subone + \Delta x^\supcor_\subone = \mu_\subone \tilde{x}_\subone - W_\subone^{-1} \Delta s^\supcor_\subone$. 
	By the fundamental theorem of calculus,
	\begin{align*}
		&\langle \primgrad(x_\subone + \Delta x^\supcor_\subone), \dualgrad(s_\subone + \Delta s^\supcor_\subone) \rangle
		= \langle \primgrad(\mu_\subone \tilde{x}_\subone - W_\subone^{-1} \Delta s^\supcor_\subone), \dualgrad(s_\subone + \Delta s^\supcor_\subone) \rangle\\
		&= \left\langle \primgrad(\mu_\subone \tilde{x}_\subone) - \int_0^1 \primhess(\mu_\subone \tilde{x}_\subone - t W_\subone^{-1} \Delta s^\supcor_\subone) W_\subone^{-1} \Delta s^\supcor_\subone \diff t, \dualgrad(s_\subone) + \int_0^1 \dualhess(s_\subone + t \Delta s^\supcor_\subone) \Delta s^\supcor_\subone \diff t \right\rangle.
	\end{align*}
	Because $\primgrad(\mu_\subone \tilde{x}_\subone) = \frac{1}{\mu_\subone} \primgrad(-\dualgrad(s_\subone)) = - \frac{1}{\mu_\subone} s_\subone$ and $\dualgrad(s_\subone) = -\dualhess(s_\subone) s_\subone$,
	\begin{align}
		&\langle \primgrad(x_\subone + \Delta x^\supcor_\subone), \dualgrad(s_\subone + \Delta s^\supcor_\subone) \rangle \nonumber\\
		&= \frac{\nu}{\mu_\subone} + \int_0^1 \left\langle \Delta s^\supcor_\subone, \left( W_\subone^{-1}  \primhess(\mu_\subone \tilde{x}_\subone - t W_\subone^{-1} \Delta s^\supcor_\subone) \dualhess(s_\subone) - \dualhess(s_\subone + t \Delta s^\supcor_\subone) \frac{1}{\mu_\subone} \right)
		s_\subone  \right\rangle \diff t \label{eq:CorrectorShadowIntermediateOne}\\
		&- \int_0^1 \int_0^1 \langle \Delta s^\supcor_\subone, W_\subone^{-1}  \primhess(\mu_\subone \tilde{x}_\subone - t W_\subone^{-1} \Delta s^\supcor_\subone) \dualhess(s_\subone + r \Delta s^\supcor_\subone) \Delta s^\supcor_\subone \rangle \diff t \diff r. \label{eq:CorrectorShadowIntermediateTwo}
	\end{align}
	Using the Cauchy-Schwartz inequality, it follows that the inner product in \eqref{eq:CorrectorShadowIntermediateOne} is at most
	\begin{equation*}
		\| \Delta s^\supcor_\subone \|_{W_\subone}^* \left\| \left( \primhess(\mu_\subone \tilde{x}_\subone - t W_\subone^{-1} \Delta s^\supcor_\subone) \dualhess(s_\subone) 
		- \frac{1}{\mu_\subone} W_\subone \dualhess(s_\subone + t \Delta s^\supcor_\subone)  \right)
		s_\subone \right\|_{W_\subone}^*,
	\end{equation*}
	where the latter norm can be further bounded by the triangle inequality:
	\begin{align}
		&\left\| \left( \primhess(\mu_\subone \tilde{x}_\subone - t W_\subone^{-1} \Delta s^\supcor_\subone) \dualhess(s_\subone) 
		- \frac{1}{\mu_\subone} W_\subone \dualhess(s_\subone + t \Delta s^\supcor_\subone)  \right)
		s_\subone \right\|_{W_\subone}^* \nonumber\\
		&\leq \frac{1}{\mu_\subone} \left\| \left( \mu_\subone \primhess(\mu_\subone \tilde{x}_\subone - t W_\subone^{-1} \Delta s^\supcor_\subone)
		-  W_\subone \right) \tilde{x}_\subone \right\|_{W_\subone}^* 
		+ \frac{1}{\mu_\subone^2} \left\| \mu_\subone \left( \dualhess(s_\subone + t \Delta s^\supcor_\subone) - \dualhess(s_\subone) \right) s_\subone \right\|_{W_\subone}, \label{eq:CorrectorShadowIntermediateOneSplit}
	\end{align}
	where we have used $\dualhess(s_\subone)s_\subone = -\dualgrad(s_\subone) = \tilde{x}_\subone$.
	
	To bound the first norm in \eqref{eq:CorrectorShadowIntermediateOneSplit} with Lemma \ref{lemma:NormToMatrixInequality}, we first note that by self-concordance and the observation that $\mu_\subone \primhess(\mu_\subone \tilde{x}_\subone) = \dualhess(s_\subone)^{-1} / \mu_\subone$,
	\begin{equation*}
		\mu_\subone \primhess(\mu_\subone \tilde{x}_\subone - t W_\subone^{-1} \Delta s^\supcor_\subone) 
		\preceq \frac{1}{(1 - t \|  W_\subone^{-1} \Delta s^\supcor_\subone \|_{\mu_\subone \tilde{x}_\subone})^2} \mu_\subone \primhess(\mu_\subone \tilde{x}_\subone)
		\preceq \frac{1}{\duallow_\subone (1 - t \|  W_\subone^{-1} \Delta s^\supcor_\subone \|_{\mu_\subone \tilde{x}_\subone})^2} W_\subone,
	\end{equation*}
	where the last step used \eqref{eq:ScalingHessianAssumptionOne}. Similarly,
	\begin{equation*}
		\mu_\subone \primhess(\mu_\subone \tilde{x}_\subone - t W_\subone^{-1} \Delta s^\supcor_\subone) 
		\succeq (1 - t \|  W_\subone^{-1} \Delta s^\supcor_\subone \|_{\mu_\subone \tilde{x}_\subone})^2 \mu_\subone \primhess(\mu_\subone \tilde{x}_\subone)
		\succeq \frac{(1 - t \|  W_\subone^{-1} \Delta s^\supcor_\subone \|_{\mu_\subone \tilde{x}_\subone})^2}{\dualup_\subone} W_\subone,
	\end{equation*}
	such that Lemma \ref{lemma:NormToMatrixInequality} and \eqref{eq:MaximumSCAidTwo} show
	\begin{align}
		\frac{1}{\mu_\subone} \left\| \left( \mu_\subone \primhess(\mu_\subone \tilde{x}_\subone - t W_\subone^{-1} \Delta s^\supcor_\subone) -  W_\subone \right) \tilde{x}_\subone \right\|_{W_\subone}^*
		&\leq \frac{1}{\mu_\subone} \| \tilde{x}_\subone \|_{W_\subone} \left\| \mu_\subone \primhess(\mu_\subone \tilde{x}_\subone - t W_\subone^{-1} \Delta s^\supcor_\subone) -  W_\subone  \right\|_{W_\subone} \nonumber\\
		&\leq \frac{\| \tilde{x}_\subone \|_{W_\subone}}{\mu_\subone} \left( \frac{1}{\duallow_\subone (1 - t \|  W_\subone^{-1} \Delta s^\supcor_\subone \|_{\mu_\subone \tilde{x}_\subone})^2} - 1 \right). \label{eq:CorrectorShadowIntermediateOneSplitOne}
	\end{align}
	
	Continuing with the second norm in \eqref{eq:CorrectorShadowIntermediateOneSplit}, note that by \eqref{eq:ScalingHessianAssumptionOne} and Lemma \ref{lemma:OrderInverseMatrix},
	\begin{align*}
		\mu_\subone [\dualhess(s_\subone + t \Delta s^\supcor_\subone) - \dualhess(s_\subone)] + W_\subone^{-1} 
		&\preceq \left( \frac{1}{(1 - t \| \Delta s^\supcor_\subone \|_{s_\subone})^2} - 1 \right) \mu_\subone \dualhess(s_\subone) + W_\subone^{-1}\\
		&\preceq \left( \frac{\dualup_\subone}{ (1 - t \| \Delta s^\supcor_\subone \|_{s_\subone})^2} - \dualup_\subone + 1 \right) W_\subone^{-1}.
	\end{align*}
	Similarly,
	\begin{align*}
		\mu_\subone [\dualhess(s_\subone + t \Delta s^\supcor_\subone) - \dualhess(s_\subone)] + W_\subone^{-1} 
		&\succeq \left( (1 - t \| \Delta s^\supcor_\subone \|_{s_\subone})^2 - 1 \right) \mu_\subone \dualhess(s_\subone) + W_\subone^{-1}\\
		&\succeq \left( \duallow_\subone (1 - t \| \Delta s^\supcor_\subone \|_{s_\subone})^2 - \duallow_\subone + 1 \right) W_\subone^{-1}.
	\end{align*}
	Hence, by Lemma \ref{lemma:NormToMatrixInequality},
	\begin{align}
		&\frac{1}{\mu_\subone^2} \left\| \mu_\subone \left( \dualhess(s_\subone + t \Delta s^\supcor_\subone) - \dualhess(s_\subone) \right) s_\subone \right\|_{W_\subone} \nonumber\\
		&= \frac{1}{\mu_\subone^2} \left\| \left( \mu_\subone [\dualhess(s_\subone + t \Delta s^\supcor_\subone) - \dualhess(s_\subone)] + W_\subone^{-1} - W_\subone^{-1} \right) s_\subone \right\|_{W_\subone} \nonumber\\
		&\leq \frac{\| s_\subone \|_{W_\subone}^* }{\mu_\subone^2} \left\| \mu_\subone [\dualhess(s_\subone + t \Delta s^\supcor_\subone) - \dualhess(s_\subone)] + W_\subone^{-1} - W_\subone^{-1} \right\|_{W_\subone^{-1}} \nonumber\\
		&\leq \frac{\| s_\subone \|_{W_\subone}^* }{\mu_\subone^2} \max \left\{ \frac{\dualup_\subone}{ (1 - t \| \Delta s^\supcor_\subone \|_{s_\subone})^2} - \dualup_\subone, \duallow_\subone - \duallow_\subone (1 - t \| \Delta s^\supcor_\subone \|_{s_\subone})^2  \right\} \nonumber\\
		&= \frac{\sqrt{\nu} }{\mu_\subone\sqrt{\mu_\subone} } \left( \frac{\dualup_\subone}{ (1 - t \| \Delta s^\supcor_\subone \|_{s_\subone})^2} - \dualup_\subone \right), \label{eq:CorrectorShadowIntermediateOneSplitTwo}
	\end{align}
	since $1/(1-\epsilon)^2 - 1 \geq 1 - (1-\epsilon)^2$ for all $\epsilon \in [0,1)$.
	In summary, the inner product in \eqref{eq:CorrectorShadowIntermediateOne} is, using the bounds \eqref{eq:CorrectorShadowIntermediateOneSplit}, \eqref{eq:CorrectorShadowIntermediateOneSplitOne}, and \eqref{eq:CorrectorShadowIntermediateOneSplitTwo}, at most
	\begin{equation}
		\label{eq:CorrectorShadowIntermediateOneBound}
		\| \Delta s^\supcor_\subone \|_{W_\subone}^* \left[ \frac{\| \tilde{x}_\subone \|_{W_\subone}}{\mu_\subone} \left( \frac{1}{\duallow_\subone (1 - t \|  W_\subone^{-1} \Delta s^\supcor_\subone \|_{\mu_\subone \tilde{x}_\subone})^2} - 1 \right) 
		+ \frac{\sqrt{\nu} }{\mu_\subone\sqrt{\mu_\subone} } \left( \frac{\dualup_\subone}{ (1 - t \| \Delta s^\supcor_\subone \|_{s_\subone})^2} - \dualup_\subone \right) \right].
	\end{equation}
	
	By the Cauchy-Schwarz inequality and the definition of the operator norm \eqref{eq:DefOperatorNorm}, the inner product in \eqref{eq:CorrectorShadowIntermediateTwo} is at most
	\begin{align}
		&\| \Delta s^\supcor_\subone \|_{W_\subone}^* \| \primhess(\mu_\subone \tilde{x}_\subone - t W_\subone^{-1} \Delta s^\supcor_\subone) \dualhess(s_\subone + r \Delta s^\supcor_\subone) \Delta s^\supcor_\subone \|_{W_\subone}^* \nonumber\\
		&\leq \frac{(\| \Delta s^\supcor_\subone \|_{W_\subone}^*)^2}{\mu_\subone^2} 
		\left\| \mu_\subone \primhess(\mu_\subone \tilde{x}_\subone - t W_\subone^{-1} \Delta s^\supcor_\subone) \right\|_{W_{\subone}} 
		\| \mu_\subone \dualhess(s_\subone + r \Delta s^\supcor_\subone) \|_{W_\subone^{-1}}.
		\label{eq:CorrectorShadowIntermediateTwoSplit}
	\end{align}
	The same technique that gave us \eqref{eq:CorrectorShadowIntermediateOneSplitOne} now yields
	\begin{align*}
		\left( \frac{(1 - t \|  W_\subone^{-1} \Delta s^\supcor_\subone \|_{\mu_\subone \tilde{x}_\subone})^2}{\dualup_\subone} + 1 \right) W_\subone 
		&\preceq \mu_\subone \primhess(\mu_\subone \tilde{x}_\subone - t W_\subone^{-1} \Delta s^\supcor_\subone) + W_\subone\\
		&\preceq \left( \frac{1}{\duallow_\subone (1 - t \|  W_\subone^{-1} \Delta s^\supcor_\subone \|_{\mu_\subone \tilde{x}_\subone})^2} +1 \right) W_\subone,
	\end{align*}
	and thus, by Lemma \ref{lemma:NormToMatrixInequality},
	\begin{equation}
		\label{eq:CorrectorShadowIntermediateTwoSplitOne}
		\left\| \mu_\subone \primhess(\mu_\subone \tilde{x}_\subone - t W_\subone^{-1} \Delta s^\supcor_\subone) \right\|_{W_{\subone}} 
		\leq \frac{1}{\duallow_\subone (1 - t \|  W_\subone^{-1} \Delta s^\supcor_\subone \|_{\mu_\subone \tilde{x}_\subone})^2}.
	\end{equation}
	Similarly, our derivation of \eqref{eq:CorrectorShadowIntermediateOneSplitTwo} also implies
	\begin{equation*}
		\left( \duallow_\subone (1 - r \| \Delta s^\supcor_\subone \|_{s_\subone})^2 + 1 \right) W_\subone^{-1}
		\preceq \mu_\subone \dualhess(s_\subone + r \Delta s^\supcor_\subone) + W_\subone^{-1} 
		\preceq \left( \frac{\dualup_\subone}{ (1 - r \| \Delta s^\supcor_\subone \|_{s_\subone})^2} + 1 \right) W_\subone^{-1},
	\end{equation*}
	and therefore Lemma \ref{lemma:NormToMatrixInequality} shows
	\begin{equation}
		\label{eq:CorrectorShadowIntermediateTwoSplitTwo}
		\| \mu_\subone \dualhess(s_\subone + r \Delta s^\supcor_\subone) \|_{W_\subone^{-1}} \leq \frac{\dualup_\subone}{ (1 - r \| \Delta s^\supcor_\subone \|_{s_\subone})^2}.
	\end{equation}
	In summary, the inner product in \eqref{eq:CorrectorShadowIntermediateTwo} is, using the bounds \eqref{eq:CorrectorShadowIntermediateTwoSplit}, \eqref{eq:CorrectorShadowIntermediateTwoSplitOne}, and \eqref{eq:CorrectorShadowIntermediateTwoSplitTwo}, at most
	\begin{equation}
		\label{eq:CorrectorShadowIntermediateTwoBound}
		\frac{(\| \Delta s^\supcor_\subone \|_{W_\subone}^*)^2 \dualup_\subone}{\mu_\subone^2 \duallow_\subone (1 - t \|  W_\subone^{-1} \Delta s^\supcor_\subone \|_{\mu_\subone \tilde{x}_\subone})^2 (1 - r \| \Delta s^\supcor_\subone \|_{s_\subone})^2}.
	\end{equation}
	
	Combining the bounds \eqref{eq:CorrectorShadowIntermediateOneBound} and \eqref{eq:CorrectorShadowIntermediateTwoBound} on the inner products in \eqref{eq:CorrectorShadowIntermediateOne} and \eqref{eq:CorrectorShadowIntermediateTwo} respectively yields
	\begin{align*}
		&\langle \primgrad(x_\subone + \Delta x^\supcor_\subone), \dualgrad(s_\subone + \Delta s^\supcor_\subone) \rangle \nonumber\\
		&\leq \frac{\nu}{\mu_\subone} + \| \Delta s^\supcor_\subone \|_{W_\subone}^* \left[ \frac{\| \tilde{x}_\subone \|_{W_\subone}}{\mu_\subone} \left( \frac{1}{\duallow_\subone (1 - \|  W_\subone^{-1} \Delta s^\supcor_\subone \|_{\mu_\subone \tilde{x}_\subone})} - 1 \right) 
		+ \frac{\sqrt{\nu} }{\mu_\subone\sqrt{\mu_\subone} } \left( \frac{\dualup_\subone}{ (1 - \| \Delta s^\supcor_\subone \|_{s_\subone})} - \dualup_\subone \right) \right]\\
		& + \frac{(\| \Delta s^\supcor_\subone \|_{W_\subone}^*)^2 \dualup_\subone}{\mu_\subone^2 \duallow_\subone (1 - \|  W_\subone^{-1} \Delta s^\supcor_\subone \|_{\mu_\subone \tilde{x}_\subone}) (1 - \| \Delta s^\supcor_\subone \|_{s_\subone})}.
	\end{align*}
	The proof is complete after a multiplication by $\mu_\subone / \nu$ and bounding the remaining norms. Lemma \ref{lemma:CorrectorProperties}\ref{item:CorrectorNorm} shows $\| \Delta s^\supcor_\subone \|_{W_\subone}^* \leq \| \primdelta_\subone \|_{W_\subone}$. Moreover, $\primhess(\mu_\subone \tilde{x}_\subone) = \dualhess(s_\subone)^{-1} / \mu_\subone^2$, so
	\begin{equation*}
		\|  W_\subone^{-1} \Delta s^\supcor_\subone \|_{\mu_\subone \tilde{x}_\subone} = \|  W_\subone^{-1} \Delta s^\supcor_\subone \|_{s_\subone}^* / \mu_\subone 
		\leq \| \Delta s^\supcor_\subone \|_{W_\subone}^* / \sqrt{ \duallow_\subone \mu_\subone} \leq \| \primdelta_\subone \|_{W_\subone} / \sqrt{ \duallow_\subone \mu_\subone},
	\end{equation*}
	by \eqref{eq:ScalingHessianAssumptionOne}. Finally, $\| \Delta s^\supcor_\subone \|_{s_\subone} \leq \sqrt{\dualup_\subone / \mu_\subone} \| \Delta s^\supcor_\subone \|_{W_\subone}^* \leq \sqrt{\dualup_\subone / \mu_\subone} \| \primdelta_\subone \|_{W_\subone}$.
\end{proof}

Note that the bound in the lemma above still depends on $\| \tilde{x}_\subone \|_{W_\subone} / \nu$. Although we chose not to analyze $\langle \tilde{x}_\subone, \tilde{s}_\subone \rangle$, we can still bound this quantity as
\begin{equation}
	\label{eq:CrudeBoundTildeX}
	\frac{\| \tilde{x}_\subone \|_{W_\subone}}{\nu} = \frac{\| \mu_\subone \tilde{x}_\subone - x_\subone + x_\subone \|_{W_\subone}}{\mu_\subone \nu} \leq \frac{\| \primdelta_\subone \|_{W_\subone} + \| x_\subone \|_{W_\subone}}{\mu_\subone \nu} = \frac{\| \primdelta_\subone \|_{W_\subone} }{\mu_\subone \nu} + \frac{1}{\sqrt{\mu_\subone \nu}}.
\end{equation}
Of course, better bounds could be found through a thorough analysis of $\langle \tilde{x}_\subone, \tilde{s}_\subone \rangle$.

\subsection{Shadow Distance After Correction}
The corrector \labelcref{eq:CorrectorDirectionG,eq:CorrectorDirectionTauKappa,eq:CorrectorDirectionXS} was designed with the goal of reducing the distance $\| \primdelta_\subone \|_{x_\subone}$. The following lemma shows that this goal is indeed attained, provided $\| \primdelta_\subone \|_{W_\subone}$ is itself not too large (in which case, $\dualup_\subone$ is close to one).

\begin{lemma}
	\label{lemma:PrimDeltaTwo}
	Let $K \subset \mathbb{R}^n$ be a proper cone admitting a $\nu$-LHSCB $\primbar$. Pick $z_\subone = (y_\subone, x_\subone, \tau_\subone, s_\subone, \kappa_\subone)$ such that $x_\subone \in \interior K$, $s_\subone \in \interior K^*$, and $\tau_\subone, \kappa_\subone > 0$.
	Assume \eqref{eq:ScalingHessianAssumptionOne} holds with $\dualup_\subone = (1+\epsilon) / \lambda$ and $\duallow_\subone = \lambda (1-\epsilon)$ for some $\lambda \in (0,1]$ and $\epsilon \geq 0$. Suppose $\| \primdelta_\subone \|_{W_\subone} / \sqrt{\mu_\subone} < \sqrt{ \min\{\primlow_\subone, 1/\primup_\subone \} }$. Then,
	\begin{align*}
		\| \primdelta_\subtwo \|_{x_\subtwo} 
		&\leq \frac{\| \primdelta_\subone \|_{W_\subone}  }{(1 - \| \primdelta_\subone \|_{W_\subone} / \sqrt{\mu_\subone \primlow_\subone})\sqrt{\mu_\subone \primlow_\subone}} \left( \frac{\dualup_\subone}{1 - \sqrt{\dualup_\subone / \mu_\subone} \| \primdelta_\subone \|_{W_\subone} } - 1 \right)\\
		&+ \frac{\| \primdelta_\subone \|_{W_\subone}^2 }{2 \mu_\subone \sqrt{\nu}} 
		\left[1 - \frac{ \| \primdelta_\subone \|_{W_\subone} }{(1 - \| \primdelta_\subone \|_{W_\subone} / \sqrt{\mu_\subone \primlow_\subone})\sqrt{\mu_\subone \primlow_\subone}} \left( \frac{\dualup_\subone}{1 - \sqrt{\dualup_\subone / \mu_\subone} \| \primdelta_\subone \|_{W_\subone} } - 1 \right) \right]^{-1}.
	\end{align*}
\end{lemma}
\begin{proof}
	By the triangle inequality, self-concordance, and \eqref{eq:ScalingHessianAssumptionOne},
	\begin{align}
		\| \primdelta_\subtwo \|_{x_\subtwo} = \| x_\subtwo - \mu_\subtwo \tilde{x}_\subtwo \|_{x_\subtwo} 
		&\leq \| x_\subtwo - \mu_\subone \tilde{x}_\subtwo \|_{x_\subtwo} + \left| \frac{\mu_\subtwo}{\mu_\subone} - 1 \right| \| \mu_\subone \tilde{x}_\subtwo \|_{x_\subtwo} \nonumber\\
		&\leq \frac{\| x_\subtwo - \mu_\subone \tilde{x}_\subtwo \|_{x_\subone}}{1 - \| \Delta x^\supcor_\subone \|_{x_\subone}} + \left| \frac{\mu_\subtwo}{\mu_\subone} - 1 \right| \frac{  \| \mu_\subone \tilde{x}_\subtwo \|_{\mu_\subone \tilde{x}_\subtwo}}{1 - \| x_\subtwo - \mu_\subone \tilde{x}_\subtwo \|_{x_\subtwo}} \nonumber\\
		&\leq \frac{\| x_\subtwo - \mu_\subone \tilde{x}_\subtwo \|_{W_\subone}}{(1 - \| \Delta x^\supcor_\subone \|_{x_\subone})\sqrt{\mu_\subone \primlow_\subone}} + \left| \frac{\mu_\subtwo}{\mu_\subone} - 1 \right| \frac{  \sqrt{\nu} }{1 - \frac{\| x_\subtwo - \mu_\subone \tilde{x}_\subtwo \|_{W_\subone}}{(1 - \| \Delta x^\supcor_\subone \|_{x_\subone})\sqrt{\mu_\subone \primlow_\subone}} }. \label{eq:CorrectorPrimDeltaIntermediate}
	\end{align}
	We continue by bounding $\| x_\subtwo - \mu_\subone \tilde{x}_\subtwo \|_{W_\subone}$. Note that by the fundamental theorem of calculus and \eqref{eq:CorrectorDirectionXS}, $x_\subtwo - \mu_\subone \tilde{x}_\subtwo$ equals
	\begin{align}
		(x_\subone + \Delta x^\supcor_\subone) - \mu_\subone (-\dualgrad(s_\subone + \Delta s^\supcor_\subone)) 
		&= x_\subone + \Delta x^\supcor_\subone + \mu_\subone \left( \dualgrad(s_\subone) + \int_0^1 \dualhess(s_\subone + t \Delta s^\supcor_\subone) \Delta s^\supcor_\subone \diff t \right) \nonumber\\
		&= \mu_\subone \tilde{x}_\subone - W_\subone^{-1} \Delta s^\supcor_\subone - \mu_\subone \tilde{x}_\subone + \int_0^1 \mu_\subone \dualhess(s_\subone + t \Delta s^\supcor_\subone) \Delta s^\supcor_\subone \diff t \nonumber\\
		&= \int_0^1 \left( \mu_\subone \dualhess(s_\subone + t \Delta s^\supcor_\subone) - W_\subone^{-1} \right) \Delta s^\supcor_\subone \diff t. \label{eq:CorrectorPrimDeltaTermOne}
	\end{align}
	To bound this expression in the norm induced by $W_\subone$, we will use the operator norm \eqref{eq:DefOperatorNorm}. 
	\begin{align}
		\left\| \left( \mu_\subone \dualhess(s_\subone + t \Delta s^\supcor_\subone) - W_\subone^{-1} \right) \Delta s^\supcor_\subone \right\|_{W_\subone}
		&\leq \| \Delta s^\supcor_\subone \|_{W_\subone}^* \sup_{\| u \|_{W_\subone}^* \leq 1} \left\| \left( \mu_\subone \dualhess(s_\subone + t \Delta s^\supcor_\subone) - W_\subone^{-1} \right) u \right\|_{W_\subone} \nonumber\\
		&= \| \Delta s^\supcor_\subone \|_{W_\subone}^* \left\| \mu_\subone \dualhess(s_\subone + t \Delta s^\supcor_\subone) - W_\subone^{-1} \right\|_{W_\subone^{-1}}.
		\label{eq:CorrectorPrimDeltaTermOneBound}
	\end{align}
	By self-concordance, \eqref{eq:ScalingHessianAssumptionOne}, and Lemma \ref{lemma:OrderInverseMatrix},
	\begin{equation*}
		\mu_\subone \dualhess( s_\subone + t \Delta s^\supcor_\subone) \preceq \frac{1}{(1 - t \| \Delta s^\supcor_\subone \|_{s_\subone} )^2} \mu_\subone \dualhess( s_\subone ) \preceq \frac{\dualup_\subone}{(1 - t \| \Delta s^\supcor_\subone \|_{s_\subone} )^2} W_\subone^{-1}.
	\end{equation*}
	Similarly,
	\begin{equation*}
		\mu_\subone \dualhess( s_\subone + t \Delta s^\supcor_\subone) \succeq (1 - t \| \Delta s^\supcor_\subone \|_{s_\subone} )^2 \mu_\subone \dualhess( s_\subone ) \succeq (1 - t \| \Delta s^\supcor_\subone \|_{s_\subone} )^2 \duallow_\subone W_\subone^{-1}.
	\end{equation*}
	It then follows from Lemma \ref{lemma:NormToMatrixInequality} that
	\begin{align*}
		\left\| \mu_\subone \dualhess(s_\subone + t \Delta s^\supcor_\subone) - W_\subone^{-1} \right\|_{W_\subone^{-1}}
		&\leq \max \left\{ \frac{\dualup_\subone}{(1 - t \| \Delta s^\supcor_\subone \|_{s_\subone} )^2} - 1, 1 - (1 - t \| \Delta s^\supcor_\subone \|_{s_\subone} )^2 \duallow_\subone \right\}\\
		&= \frac{\dualup_\subone}{(1 - t \| \Delta s^\supcor_\subone \|_{s_\subone} )^2} - 1,
	\end{align*}
	where the equality is due to the form of $\duallow_\subone$ and $\dualup_\subone$, and \eqref{eq:MaximumSCAidOne}. Hence, it follows from \eqref{eq:CorrectorPrimDeltaTermOne} and \eqref{eq:CorrectorPrimDeltaTermOneBound} that
	\begin{align}
		\| (x_\subone + \Delta x^\supcor_\subone) - \mu_\subone (-\dualgrad(s_\subone + \Delta s^\supcor_\subone))  \|_{W_\subone}
		&\leq \int_0^1 \| \Delta s^\supcor_\subone \|_{W_\subone}^* \left\| \mu_\subone \dualhess(s_\subone + t \Delta s^\supcor_\subone) - W_\subone^{-1} \right\|_{W_\subone^{-1}} \diff t \nonumber\\
		&\leq \int_0^1 \| \Delta s^\supcor_\subone \|_{W_\subone}^* \left( \frac{\dualup_\subone}{(1 - t \| \Delta s^\supcor_\subone \|_{s_\subone} )^2} - 1 \right) \diff t \nonumber\\
		&= \| \Delta s^\supcor_\subone \|_{W_\subone}^* \left( \frac{\dualup_\subone}{1 - \| \Delta s^\supcor_\subone \|_{s_\subone}} - 1 \right) \nonumber\\
		&\leq \| \primdelta_\subone \|_{W_\subone} \left( \frac{\dualup_\subone}{1 - \sqrt{\dualup_\subone / \mu_\subone} \| \primdelta_\subone \|_{W_\subone} } - 1 \right), \label{eq:CorrectorPrimDeltaIntermediateBoundOne}
	\end{align}
	where the final inequality uses Lemma \ref{lemma:CorrectorProperties}\ref{item:CorrectorNorm} and \eqref{eq:ScalingHessianAssumptionOne}.
	
	The next factor from \eqref{eq:CorrectorPrimDeltaIntermediate} we look at is $| \mu_\subtwo / \mu_\subone - 1|$. Lemma \ref{lemma:CorrectorProperties}\ref{item:CorrectorOrthogonalXS} implies that $\langle x_\subone + \Delta x^\supcor_\subone, s_\subone + \Delta s^\supcor_\subone \rangle = \langle x_\subone, s_\subone \rangle + \langle \Delta x^\supcor_\subone, \Delta s^\supcor_\subone \rangle$. 
	Therefore, it follows from Lemma \ref{lemma:CorrectorProperties}\ref{item:CorrectorXSInnerProduct} that
	\begin{equation}
		\label{eq:CorrectorPrimDeltaIntermediateBoundTwo}
		\left| \frac{\mu_\subtwo}{\mu_\subone} - 1 \right| = \frac{|\langle \Delta x^\supcor_\subone, \Delta s^\supcor_\subone \rangle |}{\langle x_\subone, s_\subone \rangle } \leq \frac{\| \primdelta_\subone \|_{W_\subone}^2}{2 \mu_\subone \nu}.
	\end{equation}
	
	The last quantity in \eqref{eq:CorrectorPrimDeltaIntermediate} that we have yet to bound is $\| \Delta x^\supcor_\subone \|_{x_\subone}$. By \eqref{eq:ScalingHessianAssumptionOne} and Lemma \ref{lemma:CorrectorProperties}\ref{item:CorrectorNorm},
	\begin{equation}
		\label{eq:CorrectorPrimDeltaIntermediateBoundThree}
		\| \Delta x^\supcor_\subone \|_{x_\subone} \leq \frac{\| \Delta x^\supcor_\subone \|_{W_\subone}}{\sqrt{\mu_\subone \primlow_\subone}} \leq \frac{\| \primdelta_\subone \|_{W_\subone}}{\sqrt{\mu_\subone \primlow_\subone}}.
	\end{equation}
	
	The proof can be completed by using \eqref{eq:CorrectorPrimDeltaIntermediateBoundOne}, \eqref{eq:CorrectorPrimDeltaIntermediateBoundTwo}, and \eqref{eq:CorrectorPrimDeltaIntermediateBoundThree} to develop \eqref{eq:CorrectorPrimDeltaIntermediate}.
\end{proof}

\section{Complexity Analysis}
\label{sec:Analysis}
In the previous two sections, we saw how the predictor and corrector step impacted the starting assumptions \labelcref{ass:XSinCone,ass:TauKappaPositive,ass:TauKappaBeta,ass:MuTildeBeta,ass:PrimDelta} for general values of the parameter $\alpha$, $\beta$, $\gamma$, and $\eta$. 
In the remainder of this work, we fix
\begin{equation}
	\label{eq:ParameterValues}
	\alpha = \frac{1}{100 \nu}, \qquad \beta = 0.9, \qquad \gamma = 0.9, \qquad \eta = \frac{1}{400 \sqrt{\nu}}.
\end{equation}
We will show that for these parameter values, the point $z_\subtwo$ satisfies \labelcref{ass:XSinCone,ass:TauKappaPositive,ass:TauKappaBeta,ass:MuTildeBeta,ass:PrimDelta} if $z$ satisfies them, and that some progress is made in moving from $z$ to $z_\subtwo$.

We note that for the values in \eqref{eq:ParameterValues}, it follows from \eqref{eq:DefNormub} and Theorem \ref{thm:ScalingBounds} that we have
\begin{equation*}
	\normub = 0.1 (\nu+1), \qquad \primlow \geq 0.97966, \qquad \dualup \leq 1.02546.
\end{equation*}
Moreover, Lemma \ref{lemma:PrimDeltaOne} and Theorem \ref{thm:ScalingBounds} show that we have
\begin{equation*}
	\| \primdelta_\subone \|_{x_\subone} \leq \frac{0.00747}{\sqrt{\nu}}, \qquad \primlow_\subone \geq 0.93719, \qquad \duallow_\subone \geq 0.92326, \qquad \primup_\subone \leq 1.06281, \qquad \dualup_\subone \leq 1.07885.
\end{equation*}

\subsection{Cone Membership}
Assumption \ref{ass:XSinCone} states that $x \in \interior K$ and $s \in \interior K^*$. By Theorem \ref{thm:PredictorNorms}, $\| \alpha \Delta x^\suppred \|_x \leq \alpha \sqrt{\theta/(\beta \primlow)}$, which, for the parameter values in \eqref{eq:ParameterValues} and $\nu \geq 1$, is at most $0.01 \sqrt{0.1(\nu+1)/0.88170}/ \nu \leq 0.00477 < 1$. Hence, $x_\subone \in \interior K$.
Similarly, Theorem \ref{thm:PredictorNorms} shows $\| \alpha \Delta s^\suppred \|_s \leq \alpha \sqrt{\theta \dualup/\beta } \leq 0.01 \sqrt{0.1(\nu+1)1.1394}/ \nu \leq 0.00478 < 1$, implying that $s_\subone \in \interior K^*$.

Next, we consider the corrector step. It follows from \eqref{eq:ScalingHessianAssumptionOne} and Lemma \ref{lemma:CorrectorProperties}\ref{item:CorrectorNorm} that
\begin{equation*}
	\primlow_\subone \mu_\subone \| \Delta x^\supcor_\subone \|_{x_\subone}^2  + \frac{\mu_\subone}{\dualup_\subone} \| \Delta s^\supcor_\subone \|_{s_\subone}^2 
	\leq \| \Delta x^\supcor_\subone \|_{W_\subone}^2  + (\| \Delta s^\supcor_\subone \|_{W_\subone}^*)^2
	\leq \| \primdelta_\subone \|_{W_\subone}^2
	\leq \primup_\subone \mu_\subone \| \primdelta_\subone \|_{x_\subone}^2
\end{equation*}
Hence, $\| \Delta x^\supcor_\subone \|_{x_\subone} \leq \| \primdelta_\subone \|_{x_\subone} \sqrt{\primup_\subone / \primlow_\subone} \leq 0.00796 < 1$ and $\| \Delta s^\supcor_\subone \|_{s_\subone} \leq \| \primdelta_\subone \|_{x_\subone} \sqrt{ \primup_\subone \dualup_\subone } \leq 0.00800 < 1$. In conclusion, $x_\subtwo \in \interior K$ and $s_\subtwo \in \interior K^*$.

\subsection{Tau and Kappa Relation to Complementarity}
We continue by verifying \ref{ass:TauKappaBeta} for $z_\subtwo$. By Lemma \ref{lemma:PredictorTauKappa} and Lemma \ref{lemma:CorrectorProperties}\ref{item:CorrectorTauKappa},
\begin{equation}
	\label{eq:TauKappaLowerBoundIntermediate}
	\tau_\subtwo \kappa_\subtwo \geq \frac{\mue_\subone}{1 - \alpha(1-\gamma)} \left( \beta (1-\alpha) + \alpha \gamma - \tfrac{1}{2} \alpha^2 \normub \right) - \tfrac{1}{2} \| \primdelta_\subone \|_{W_\subone}^2.
\end{equation}
Since $\mue_\subone = \mue_\subtwo$ by Lemma \ref{lemma:CorrectorProperties}\ref{item:CorrectorComplementarity}, our next step is to bound $\| \primdelta_\subone \|_{W_\subone}^2$ in terms of $\mue_\subtwo$.
To this end, let $\omega_1$ be the upper bound on $\mu_\subone/\mu$ from Corollary \ref{cor:MuOneOverMu}. For the parameter values \eqref{eq:ParameterValues}, we have
\begin{equation*}
	\frac{\mu_\subone}{\mu} \leq \omega_1 \equiv  1 + \frac{\alpha^2 \normub}{2 \beta \nu} + \alpha \left(  \frac{\gamma }{\beta} - 1 \right) = 1 + \frac{\nu+1}{180000 \nu^3} \leq 1 + \frac{1}{90000}.
\end{equation*}
By \eqref{eq:ScalingHessianAssumptionOne}, Lemma \ref{lemma:MuMueBounds}, and
Lemma \ref{lemma:PredictorProperties}\ref{item:PredictorComplementarity},
\begin{equation}
	\label{eq:TauKappaLowerBoundIntermediateTwo}
	\| \primdelta_\subone \|_{W_\subone}^2
	\leq \primup_\subone \mu_\subone \| \primdelta_\subone \|_{x_\subone}^2
	\leq \primup_\subone \mu \omega_1 \| \primdelta_\subone \|_{x_\subone}^2
	\leq \primup_\subone (2-\beta) \mue \omega_1 \| \primdelta_\subone \|_{x_\subone}^2
	= \frac{\primup_\subone (2-\beta)  \omega_1}{1 - \alpha(1-\gamma)} \mue_\subone \| \primdelta_\subone \|_{x_\subone}^2.
\end{equation}
Since for all $\nu \geq 1$,
\begin{equation*}
	\frac{0.1(\nu+1)}{20000 \nu} + \frac{1.1 \primup_\subone \omega_1 \| \primdelta_\subone \|_{x_\subone}^2}{2} \nu 
	\leq \frac{1}{10^5} + 0.55 \times 1.06281 \times \frac{90001}{90000} \times \left(  \frac{0.00747}{\sqrt{\nu}} \right)^2 \nu < \frac{0.09}{100},
\end{equation*}
we have 
\begin{equation}
	\label{eq:TauKappaLowerBoundIntermediateThree}
	\tfrac{1}{2} \alpha^2 \normub + \tfrac{1}{2} \primup_\subone (2-\beta) \omega_1 \| \primdelta_\subone \|_{x_\subone}^2 \leq (1-\beta) \alpha \gamma. 
\end{equation}
Thus, plugging \labelcref{eq:TauKappaLowerBoundIntermediateTwo,eq:TauKappaLowerBoundIntermediateThree} into \eqref{eq:TauKappaLowerBoundIntermediate}, we see that
\begin{align*}
	\tau_\subtwo \kappa_\subtwo 
	&\geq \frac{\beta (1-\alpha) + \alpha \gamma - \tfrac{1}{2} \alpha^2 \normub - \tfrac{1}{2} \primup_\subone (2-\beta) \omega_1 \| \primdelta_\subone \|_{x_\subone}^2}{1 - \alpha(1-\gamma)} \mue_\subtwo\\
	&\geq \frac{\beta (1-\alpha) + \alpha \gamma - (1-\beta) \alpha \gamma}{1 - \alpha(1-\gamma)} \mue_\subtwo = \beta \mue_\subtwo,
\end{align*}
which shows that \ref{ass:TauKappaBeta} is satisfied for $z_\subtwo$.

\subsection{Tau and Kappa Positive}
To verify that \ref{ass:TauKappaPositive} holds for $z_\subtwo$, we first turn to Lemma \ref{lemma:PredictorTauKappa}. It is not hard to verify that for the values in \eqref{eq:ParameterValues},
\begin{equation*}
	\frac{\gamma - \beta + \sqrt{(\beta-\gamma)^2 + 2 \beta \normub}}{\normub} = \sqrt{\frac{2 \beta}{\theta}} = \sqrt{\frac{18}{\nu+1}} > \frac{1}{100 \nu} = \alpha,
\end{equation*}
so $\tau_\subone > 0$ and $\kappa_\subone > 0$. Lemma \ref{lemma:CorrectorProperties}\ref{item:CorrectorTauKappa} shows that $\tau_\subtwo > 0$ and $\kappa_\subtwo > 0$ as well, since it was shown above that $\tau_\subtwo \kappa_\subtwo \geq \beta \mue_\subtwo > 0$.

\subsection{Shadow Inner Product}
The next assumption we look at is \ref{ass:MuTildeBeta}. Lemma \ref{lemma:ShadowInnerProduct} gives us an upper bound on $\mu_\subone \tilde{\mu}_\subtwo$, while we want an upper bound on $\beta \mue_\subtwo \tilde{\mu}_\subtwo$. Note that by Lemmas \ref{lemma:CorrectorProperties}\ref{item:CorrectorComplementarity} and \ref{lemma:PredictorProperties}\ref{item:PredictorComplementarity}
\begin{equation}
	\label{eq:ShadowProductBoundIntermediate}
	\beta \mue_\subtwo \tilde{\mu}_\subtwo 
	= \beta (1- \alpha(1-\gamma)) \mue \tilde{\mu}_\subtwo 
	= \beta (1- \alpha(1-\gamma)) \frac{\mue}{\mu_\subone} \mu_\subone \tilde{\mu}_\subtwo. 
\end{equation}
To find an upper bound on $\mue / \mu_\subone$, note that by definition,
\begin{equation*}
	\frac{\mue}{\mu_\subone} 
	= \frac{\mue \nu}{ \langle x + \alpha \Delta x^\suppred, s + \alpha \Delta s^\suppred \rangle }
	= \frac{\mue \nu}{ \mu \nu + \alpha \langle x, W \Delta x^\suppred + \Delta s^\suppred \rangle + \alpha^2 \langle \Delta x^\suppred, \Delta s^\suppred \rangle }
\end{equation*}
It follows from \labelcref{eq:AffineDirectionXS,eq:CenteringDirectionXS,eq:DefPredictor} that $W \Delta x^\suppred + \Delta s^\suppred = - s + \gamma \mue \tilde{s}$. Hence, by Corollary \ref{cor:InnerProductPredictorBound} and Lemma \ref{lemma:MuMueBounds},
\begin{equation}
	\label{eq:ShadowProductBoundIntermediateTwo}
	\frac{\mue}{\mu_\subone} 
	\leq \frac{\mue \nu}{ \mu \nu + \alpha (-\mu \nu + \gamma \mue \nu ) - \tfrac{1}{2} \alpha^2 \mue \theta } 
	\leq \frac{\mue \nu}{ (1-\alpha) \beta \mue \nu + \alpha \gamma \mue \nu  - \tfrac{1}{2} \alpha^2 \mue \theta }
	= \frac{1}{ (1-\alpha) \beta + \alpha \gamma  - \tfrac{1}{2} \alpha^2 \theta/ \nu }.
\end{equation}
Next, note that $\| \primdelta_\subone \|_{W_\subone} \leq \sqrt{\primup_\subone \mu_\subone} \| \primdelta_\subone \|_{x_\subone} \leq 0.00771 \sqrt{\mu_\subone / \nu} $ by \eqref{eq:ScalingHessianAssumptionOne}.
Consequently, by Lemma \ref{lemma:ShadowInnerProduct} and \eqref{eq:CrudeBoundTildeX},
\begin{equation}
	\label{eq:ShadowProductBoundIntermediateThree}
	\beta (1-\alpha(1-\gamma)) ( \mu_\subone \tilde{\mu}_\subtwo - 1) \leq \frac{0.00077}{\nu} 
	\leq \frac{0.9}{1000 \nu} - \frac{\nu+1}{200 000 \nu^3} 
	= \alpha \gamma (1-\beta) - \frac{\alpha^2 \theta}{2 \nu}.
\end{equation}
Plugging \labelcref{eq:ShadowProductBoundIntermediateTwo,eq:ShadowProductBoundIntermediateThree} into \eqref{eq:ShadowProductBoundIntermediate} yields
\begin{equation*}
	\beta \mue_\subtwo \tilde{\mu}_\subtwo 
	\leq \frac{\beta(1- \alpha(1-\gamma))}{(1-\alpha) \beta + \alpha \gamma  - \tfrac{1}{2} \alpha^2 \theta/ \nu} \mu_\subone \tilde{\mu}_\subtwo
	\leq \frac{\beta (1- \alpha(1-\gamma)) + \alpha \gamma (1-\beta) - \tfrac{1}{2} \alpha^2 \theta/ \nu}{(1-\alpha) \beta + \alpha \gamma  - \tfrac{1}{2} \alpha^2 \theta/ \nu} = 1.
\end{equation*}

\subsection{Shadow Distance}
The final assumption to consider is \ref{ass:PrimDelta}. It follows from Lemma \ref{lemma:PrimDeltaTwo} that for the parameter values in \eqref{eq:ParameterValues},
\begin{equation*}
	\| \primdelta_\subtwo \|_{x_\subtwo} \leq \frac{0.00074}{\sqrt{\nu}} < \frac{0.0025}{\sqrt{\nu}} = \frac{1}{400 \sqrt{\nu}} = \eta.
\end{equation*}

\subsection{Complexity Result}
The preceding sections demonstrated that for the parameter values \eqref{eq:ParameterValues}, an iteration of Algorithm \ref{alg:MOSEKNonSymmetric} starting at a solution $z$ satisfying \labelcref{ass:XSinCone,ass:TauKappaPositive,ass:TauKappaBeta,ass:MuTildeBeta,ass:PrimDelta} ends at a solution $z_\subtwo$ for which \labelcref{ass:XSinCone,ass:TauKappaPositive,ass:TauKappaBeta,ass:MuTildeBeta,ass:PrimDelta} also hold. By Lemmas \ref{lemma:PredictorProperties}\ref{item:PredictorG}, \ref{lemma:PredictorProperties}\ref{item:PredictorComplementarity}, \ref{lemma:CorrectorProperties}\ref{item:CorrectorG}, and \ref{lemma:CorrectorProperties}\ref{item:CorrectorComplementarity}, the residuals and complementarity reduce at a rate of $1-\alpha(1-\gamma)$ per iteration. 
If we start at an initial point $\zinit$ as in \eqref{eq:DefInitialPoint}, then $\mue = 1$ and \labelcref{ass:XSinCone,ass:TauKappaPositive,ass:TauKappaBeta,ass:MuTildeBeta,ass:PrimDelta} all hold. Thus, after $k$ iterations of Algorithm \ref{alg:MOSEKNonSymmetric}, the complementarity gap is exactly $(1 - \alpha(1-\gamma))^k$, and the residuals are $G(\zinit) (1 - \alpha(1-\gamma))^k$. This gives rise to our main result.

\begin{theorem}
	\label{thm:ComplexityResult}
	Let $K \subset \mathbb{R}^n$ be a proper cone admitting a $\nu$-LHSCB. Assume we start Algorithm \ref{alg:MOSEKNonSymmetric} from a starting point satisfying \eqref{eq:DefInitialPoint}, with parameter values as in \eqref{eq:ParameterValues}. For any $\varepsilon \in (0,1)$, the algorithm produces a solution $z$ satisfying
	\begin{equation*}
		\mue \leq \varepsilon \qquad \text{and} \qquad \| G(z) \| \leq \varepsilon \| G(\zinit) \|,
	\end{equation*}
	in $O(\nu \log(1/\varepsilon))$ iterations.
\end{theorem}


\section{Concluding Remarks}
We have shown that Algorithm \ref{alg:MOSEKNonSymmetric} computes an approximate solution to the homogeneous model in polynomial time. This gives a theoretical foundation to the algorithm by Dahl and Andersen \cite{dahl2019primal}, as implemented in MOSEK. While our scaling matrix, predictor direction, and neighborhood assumptions \labelcref{ass:MuTildeBeta,ass:TauKappaBeta,ass:TauKappaPositive,ass:XSinCone} are all taken from Dahl and Andersen, some differences remain. Most notably, MOSEK uses a higher-order corrector that depends on the third derivative of the primal barrier, where we have used \labelcref{eq:CorrectorDirectionG,eq:CorrectorDirectionTauKappa,eq:CorrectorDirectionXS}. We leave the analysis of this higher-order corrector for future research.

\section*{Acknowledgments}
The authors would like to thank Etienne de Klerk for various comments that helped to improve the presentation of this paper.

\bibliographystyle{abbrv}
\bibliography{Bibliography}

\end{document}